\documentclass{amsart}
\usepackage{amsmath,amssymb,amsthm, mathrsfs, mathtools, bm}
\usepackage{amscd,mathtools}
\usepackage[shortlabels]{enumitem}
\usepackage[utf8]{inputenc}
\usepackage[english]{babel}
\usepackage{color}
\usepackage{hyperref}
\usepackage{cite}
\usepackage{float}
\usepackage{comment}
\usepackage{mathtools}
\usepackage{csquotes}
\usepackage[normalem]{ulem}
\usepackage{caption, subcaption}

\usepackage{tikz}
\usetikzlibrary{automata}

\usepackage{tikz-cd}
\usetikzlibrary{calc,decorations.pathreplacing}
\usetikzlibrary{arrows}
\usetikzlibrary{decorations.pathreplacing}
\usetikzlibrary{intersections}

 \newtheorem{theorem}{Theorem}[section]
  \newtheorem{proposition}[theorem]{Proposition}
  \newtheorem{corollary}[theorem]{Corollary}
  \newtheorem{lemma}[theorem]{Lemma}

    \newtheorem{question}{Question}

  \newtheorem{introthm}{Theorem}
  \newtheorem{introcor}[introthm]{Corollary}

  \theoremstyle{definition}
  \newtheorem{definition}[theorem]{Definition}

  \newtheorem*{claim*}{Claim}
  
  \newtheorem{example}[theorem]{Example}

  \newtheorem*{answer*}{Answer}
  \newtheorem*{application*}{Application}

  \theoremstyle{remark}
  \newtheorem{remark}[theorem]{Remark}
  \newtheorem*{remark*}{Remark}

\DeclarePairedDelimiterX{\Norm}[1]{\lVert}{\rVert}{#1}
\theoremstyle{definition}

  \newcommand{\sC}{{\sf C}}

  \newcommand{\sQ}{{\sf Q}}

  \newcommand{\cc}{{\sf c}}   
  \newcommand{\dd}{{\sf d}}     

  \newcommand{\mm}{{\sf m}}   
  \newcommand{\nn}{{\sf n}}

  \newcommand{\qq}{{\sf q}}   
    
    \newcommand{\RR}{{\mathbb{R}}}

  \newcommand{\gothic}{\mathfrak}
  \newcommand{\go}{{\gothic o}}

  \newcommand{\calH}{\mathcal{H}}

  \newcommand{\calN}{\mathcal{N}}

\newcommand{\xspace}{\,}

\newcommand{\CAT}{\ensuremath{\operatorname{CAT}(0)}\xspace}  
  \newcommand{\ST}{\mathbin{\Big|}} 
 
\newcommand{\sg}{\mathfrak{g}} 

\newcommand{\ou}{%
  \mathrel{%
    \vcenter{\offinterlineskip
      \ialign{##\cr$+$\cr\noalign{\kern-1.5pt}$-$\cr}%
    }%
  }%
}

\newcommand{\myGlobalTransformation}[2]
{
    \pgftransformcm{0.2}{0.8}{0}{1}{\pgfpoint{#1}{#2}}
}

\begin{document}

\title[Sublinearly Morse geodesics in CAT(0) spaces]{Sublinearly Morse geodesics in CAT(0) spaces: lower divergence and hyperplane characterization}


  \author   {Devin Murray}
 \address{Department of Mathematics, University of Hawaii at Manoa, Hawaii, Hawaii }
  \email{murray@math.hawaii.edu}
  
 \author   {Yulan Qing}
 \address{Shanghai Center for Mathematical Sciences, Fudan University, Shanghai }
 \email{yulan.qing@gmail.com}
 
  \author   {Abdul Zalloum}
 \address{Department of Mathematics, Queen's University, Kingston, ON }
 \email{az32@queensu.ca}

 

\begin{abstract}
We introduce the notion of $\kappa$-lower divergence for geodesic rays in CAT(0) spaces. Building on work of Charney and Sultan we give various characterizations of $\kappa$-contracting geodesic rays using $\kappa$-lower divergence and $\kappa$-slim triangles. We also characterize $\kappa$-contracting geodesic rays in CAT(0) cube complexes using sequences of well-separated hyperplanes.
\end{abstract}

\maketitle

\section{Introduction}

The Gromov boundary \cite{Gro87} has been widely used in studying algebraic, geometric, and dynamic properties of hyperbolic groups (surveyed by Kapovich and Benakli in \cite{Kapovich}). In recent years, a significant amount of effort has been put into studying several different notions of ``hyperbolic-like" geodesics which arise naturally in spaces which are not $\delta$-hyperbolic, but still captures some of the essence of the coarse geometry of $\delta$-hyperbolic spaces. 

Three well-studied properties for geodesic rays in hyperbolic spaces are \emph{super linear lower-divergence}, the \emph{uniformly contracting} property, and the \emph{$\delta$-slim} property. Charney and Sultan showed that, in a CAT(0) space, these properties all characterize Morse geodesics \cite{ChSu2014}. These properties are all widely used tools to study CAT(0) groups (see \cite{caprice}, \cite{H09}, \cite{Ballmann2008},\cite{BF2009}, \cite{Murray2019}, \cite{Behrstock2011} for a sampling). One of the downsides is that contracting geodesics are generally quite rare regardless of the method used to measure them, for the details on random walks see Cordes, Dussaule and Gekhtman\cite{CDG}.

A weaker notion of a ``hyperbolic-like" geodesic is that of a \emph{$\kappa$-contracting geodesic}. Let $\kappa$ be a sublinear function, a geodesic ray $b$ is said to be $\kappa$-contracting if there exists a constant $\cc \geq 0$ such that projections of any disjoint ball $B_r(x)$ to $b$ have diameter bounded above by $\cc \kappa(||x||),$ where $||x||=d(b(0),x).$ In \cite{QRT19}, Qing, Tiozzo and Rafi showed that the collection of $\kappa$-contracting geodesic rays can be used to define a boundary at infinity with some nice properties. In particular, the $\kappa$-contracting boundary serves as a topological model of the Poisson boundary for right-angled Artin groups. Furthermore, it was recently announced by Gekhtman, Qing and Rafi \cite{GQR} that $\kappa$-contracting geodesics are generic in rank-1 CAT(0) spaces in various natural measures.

However, $\kappa$-contracting geodesics are still not very well studied. 

\subsection{Statement of Results.}

In  this  paper  we  provide several new characterizations of $\kappa$-contracting geodesics. Let $b$ be a geodesic ray and fix some  $r>0$ and then $t > r \kappa(t)$. Let $\rho_\kappa (r,t)$ denote the infimum of the lengths of all paths from $b(t-r\kappa(t))$ to $b(t+r\kappa(t))$ which lie outside the open ball of radius $r\kappa(t)$ about $b(t)$. See Figure \ref{fig: definition of lower divergence intro}. Given such a geodesic ray $b$, we define the $\kappa$-\emph{lower divergence} of $b$ to be growth rate of the function $$div_{\kappa}(r):= \underset{t>r\kappa(t)}{\text{inf}}\,\,\frac{\rho_\kappa(r,t)}{\kappa(t)}.$$ 

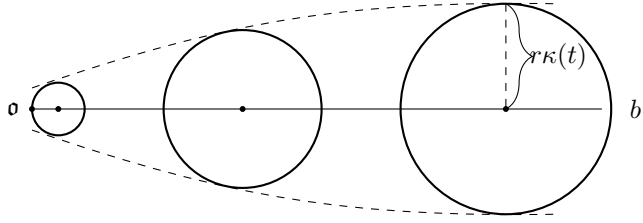
\begin{figure}[h!]
\begin{tikzpicture}[scale=0.7]
 \tikzstyle{vertex} =[circle,draw,fill=black,thick, inner sep=0pt,minimum size=.5 mm]
[thick, 
    scale=1,
    vertex/.style={circle,draw,fill=black,thick,
                   inner sep=0pt,minimum size= .5 mm},
                  
      trans/.style={thick,->, shorten >=6pt,shorten <=6pt,>=stealth},
   ]

  \node[vertex] (o) at (0,0)[label=left:$\go$] {}; 
  \node(a) at (11,0)[label=right:$b$] {}; 
  
  \draw (o)--(a){};

  \draw [dashed] (0, 0.4) to [bend left = 10] (10,2){};
  
    \draw [dashed] (0, -0.4) to [bend right = 10] (10,-2){};
\draw [thick] (0.5,0) circle (0.5cm);
\draw [thick] (4,0) circle (1.5cm);
\draw [thick] (9,0) circle (2cm);
\node[vertex] at (0.5, 0){}; 
\node[vertex] at (4, 0){}; 
\node[vertex] at (9, 0){};

  \draw [decorate,decoration={brace,amplitude=10pt},xshift=0pt,yshift=0pt]
  (9, 2) -- (9,0)  node [thick, black,midway,xshift=0pt,yshift=0pt] {};       
  \draw [dashed]  (9, 2) -- (9,0){};
 \node at (10,1) {$ r\kappa(t)$};

  \end{tikzpicture}
  
\caption{Definition of $\kappa$-lower divergence.}
\label{fig: definition of lower divergence intro}
\end{figure}

\begin{introthm} \label{thm: first intro thm}
For a geodesic ray $b$ in a CAT(0) space. The following are equivalent:
\begin{enumerate}
    \item The geodesic $b$ is $\kappa$-contracting.
    \item The $\kappa$-lower divergence of $b$ is superlinear.
    \item The $\kappa$-lower divergence of $b$ is at least quadratic.
    \item The geodesic $b$ is $\kappa$-slim.
     
\end{enumerate}

\end{introthm}

We also give a combinatorial characterization of $\kappa$-contracting  geodesic rays in the context of CAT(0) cube complexes. In this setting, the combinatorial structure of a CAT(0) cube complex makes it possible to characterize the Morse and contracting properties of a geodesic ray without reference to projections of balls or wandering quasi-geodesic segments. Two disjoint hyperplanes $h_1, h_2$ are said to be \emph{$k$-well-separated} if each collection of hyperplanes intersecting both $h_1, h_2$, which contains no facing triples, has at most cardinality $k$. A facing triple is a collection of three disjoint hyperplanes, none of which separates the other two.

\begin{introthm} \label{thm: 3rd theorem in the intro}
Let $X$ be a locally finite CAT(0) cube complex. A geodesic ray $b \in X$ is $\kappa$-contracting if and only if there exists $\cc>0$ such that $b$ crosses an infinite sequence of hyperplanes $h_1, h_2,...$ at points $b(t_i) = b \cap h_{i}$ satisfying: 

\begin{enumerate}
    \item $d(t_i,t_{i+1}) \leq \cc \kappa(t_{i+1}).$
    \item $h_i,h_{i+1}$ are $\cc \kappa(t_{i+1})$-well-separated.
\end{enumerate}

Furthermore, if $X$ is a cocompact CAT(0) cube complex which has a factor system in the sense of \cite{Behrstock2017}, then condition (2) can be replaced with a uniform constant $\cc'$.

\end{introthm}

We remark that Theorem \ref{thm: 3rd theorem in the intro} does not assume that $X$ is uniformly locally finite.  
In \cite{HypInCube}, Genevois introduced a collection of Gromov hyperbolic metric spaces $(Y_n, d_n)$ which generalize the contact graph, see Subsection \ref{subsec: Gen's graph}. As an application to Theorem \ref{thm: 3rd theorem in the intro}, we show the following.

\begin{introcor}\label{cor:final statement of the intro}  Let $X$ be a cocompact CAT(0) cube complex with a factor system. There exists a Gromov hyperbolic space $(Y,d_Y)$ and a projection map $p:X \rightarrow Y$ such that for any $\kappa$-contracting geodesic ray $b$, we have  $d_Y(p(b(0)), p(b(t))) \geq \frac{t}{\dd_1\kappa(t)}-\dd_2$ for some constants $\dd_1,\dd_2$ depending only on $b$. In particular, $\kappa$-contracting geodesic rays project to infinite diameter subsets of $Y.$ Furthermore, when $X$ is the Salvetti complex of a right-angled Artin group, the space $Y$ is the contact graph of $X$.

\end{introcor}

%
%
%

\subsection{Separated vs well-separated hyperplanes}
In a \CAT cube complex there are many kinds of ``separation'' type properties. Two hyperplanes $h_1,h_2$ are said to be $k$-separated if the number of hyperplanes crossing them both is bounded above by $k$. They are $k$-well-separated if any set of hyperplanes crossing both, but contains no facing triples, has size bounded above by $k$. The notion of well-separated hyperplanes was introduced by Genevois, for example, see \cite{Genevois2020}.

A geodesic ray $b$ is said to be \emph{$D$-uniformly contracting} if there exists a constant $D$ such that projections of disjoint balls to $b$ have diameters uniformly bounded above by $D$. In \cite{ChSu2014}, Charney and Sultan characterized uniformly contracting geodesic rays as those which cross an infinite sequence of hyperplanes $\{h_i\}^{\infty}_{i=1}$ at points $t_i$ where $d(b(t_i), b(t_{i+1}))$ is uniformly bounded above by a constant $r$ and each two consecutive hyperplanes $h_i,h_{i+1}$ are uniformly $k$-separated for some $k \geq0$.

A natural question to ask is whether $\kappa$-contracting geodesic rays can be similarly described using the simpler $k$-separated notion. In other words, it is reasonable to expect that $\kappa$-contracting geodesic rays are characterized as ones which cross an infinite sequence of hyperplanes $\{h_i\}_{i=1}^{\infty}$ at points $t_i$ where $d(b(t_i), b(t_{i+1}))$ grows like $\kappa(t_i)$ and the separation of each two consecutive hyperplanes $h_i,h_{i+1}$ grows like $\kappa(t_i)$. This is, in fact, not the case. We give an example (Example~\ref{example: counter example}) of a $\log_2$-contracting geodesic ray where the separation grows linearly (in particular, faster than every sublinear function). More precisely, we show the following.

\begin{introthm}\label{theorem:counter example in the intro}
There is a uniformly locally finite CAT(0) cube complex $X$ and a $\log_2$-contracting geodesic ray $b \subset X$ such that the following holds: For any sublinear function $\kappa$, and any collection of hyperplanes $\{h_i\}_{i=1}^{\infty}$ crossing $b$ at $t_i$, if $d(b(t_{i-1}),b(t_{i})) \leq \cc \kappa(t_i)$ for some $\cc \geq 0$, then $$\displaystyle \lim_{t_i \to \infty} \frac{\#\{ \text{hyperplanes }h \; | \; h \text{ intersects } h_{i-1} \text{ and } h_i \}}{t_i} \geq  \frac{1}{2}.$$ In particular, the number of hyperplanes crossing both $h_{i-1}, h_{i}$ grows faster than every sublinear function of $t_i.$
\end{introthm}

To understand the difference between $k$-separation and $k$-well-separation it is useful to see how these are related to the volume and diameter of a convex set. For example, consider a ball of radius $D$ given by $B_D$ in the Cayley graph of $F_2$. The number of hyperplanes meeting such a ball is about $4.3^D$ while the diameter of this set is $2D$. Notice however, that the number of hyperplanes meeting $B_D$ which contain no facing triple is $2D$. In a quantifiable way, the number of hyperplanes that meet a convex set measures the sets volume, while the number of hyperplanes that meet the set \emph{and} contain no facing triples measures the diameter.

Since the $\kappa$-contracting property is a statement about the diameter of projections and not the volume, the notion of well-separation will be better suited for our characterization (in the special case of uniform contraction, bounding the diameter of the projection is equivalent to bounding its volume and the separation of hyperplanes is the simpler combinatorial property to use, hence why it was used in \cite{ChSu2014}).

For a slightly less trivial example, let $B_D$ be as above and consider the space $C_D=B_D \times \mathbb{R}^{+}$. Let $b$ be the geodesic ray starting in $C_D$ at $e$ and going in the $\mathbb{R}^{+}$ direction of $C_D$. This geodesic ray is $D$-uniformly contracting while every two hyperplanes $h_1,h_2$ crossed by $b$ are $4.3^D$-separated. Nonetheless, $h_1,h_2$ are $D$-well-separated. As illustrated in this example, the well-separation notion inter-plays with contraction in a linear fashion whereas separation does not. Since we are trying to understand $\kappa$-contracting geodesic rays where the contraction grows as the sublinear function $\kappa$, the well-separation notion is the appropriate one to use.

\subsection{History}  

Divergence rose in the study of non-positively curved manifolds and metric spaces as a measure of how fast geodesic rays travel away from each other. In particular, Gromov conjectured that all pairs of geodesic rays in the universal cover of a closed Riemannian manifold of non-positive curvature diverge either linearly or exponentially \cite{Grom}. Gersten \cite{Ger} provided the first examples of CAT(0) spaces whose divergence did not satisfy the linear/exponential dichotomy and showed that such examples are closely tied to other areas in mathematics. Duchin and Rafi \cite{DuchinRafi} show that both Teichmuller space (with the Teichmuller metric) and the mapping class group (with a word metric) have geodesic divergence that is intermediate between the linear rate of flat spaces and the exponential rate of hyperbolic spaces. Algom-Kfir \cite{AlgomKfir2011} shows the divergence function in $CV_{n}$ is at least quadratic. In \cite{Levcovitz2018}, Levcovitz gave a classification of all right-angled Coxeter groups with quadratic divergence and showed right-angled Coxeter groups cannot exhibit a divergence function between quadratic and cubic. More recently, Arzhantseva-Cashen-Gruber-Hume \cite{Cashen2016}  shows that Morse geodesic rays in proper metric spaces have  superlinear divergences. 

Another motivation for studying collections of geodesic rays is to introduce a boundary at infinity. The topology and dynamics of these boundaries often provides information about the geometry of the space and the algebraic behaviour of groups acting on the space.  For CAT(0) spaces, there are a number of boundaries, the visual boundary, the Tits boundary, and more recently the contracting boundary and the $\kappa$-contracting boundary (see \cite{Merlin} by Incerti-Medici on comparing such boundaries). Each boundary has its own advantages and disadvantages, the visual and Tits boundaries give a lot of information about the product structure and maximal flats in a space, however, neither of them are quasi-isometric invariants as shown by Croke and Kleiner \cite{CKexample}. This limits their use in studying some aspects of CAT(0) groups. The contracting boundary (\cite{ChSu2014} for \CAT space and \cite{Cordes} by Cordes for proper metric spaces) is a quasi-isometry invariant, however it is not a metrizable space, and in a probabilistic sense contracting geodesic rays are rare in CAT(0) spaces that contain flats \cite{CDG}. The $\kappa$-boundary was introduced by Qing, Rafi and Tiozzo in \cite{QRT19} and sought to rectify some of the shortcomings found in the these other boundaries. They show that the $\kappa$-boundary is metrizable and invariant under quasi-isometries. Moreover, in \cite{QZ19} Qing and Zalloum further studied the topology and dynamics of this boundary for \CAT spaces. For example, they show that the $\kappa$-boundary is a strong visibility space and that rank one isometries have dense axes in the $\kappa$-boundary. In fact, every proper geodesic space has a $\kappa$-boundary that is QI-invariant and metrizable, and in the case of mapping class groups it can serve as Poisson boundaries as well \cite{QRT20}.

\subsection*{Further Questions}

\begin{question} 
\emph{For CAT(0) spaces, the notion of $\kappa$-Morse and $\kappa$-contracting are equivalent, but they are different in general. For a proper geodesic metric space can sublinearly Morse (quasi)-geodesics be characterized in terms of having superlinear $\kappa$-lower divergence?}

\end{question}

\begin{question}
\emph{Can one show the following. Let $X$ be a CAT(0) cube complex and let $b \in X$ be a geodesic ray. If no hyperplane $h$ has an infinite projection on $b$, then $b$ is sublinearly contracting ($\kappa$-contracting for some $\kappa$) if and only if it makes a sublinear progress in the contact graph of $X$}.
\end{question}

\begin{question}
 \emph{We claim that Corollary \ref{cor:final statement of the intro} implies the existence of a map $i: \partial_\kappa X \rightarrow \partial Y$ between the $\kappa$-boundary of a cocompact CAT(0) cube complex $X$ with a factor system to the Gromov boundary of the corresponding hyperbolic space $Y$. Using different tools  \cite{Abbot} have shown that an analogous map the Morse boundary of a hierarchically hyperbolic space and the boundary of a hyperbolic space $Y$. Further, they showed that this map is continuous and 1-1 \cite{Abbot}. Can one show that $i: \partial_\kappa X \rightarrow \partial Y$ is a well-defined continuous injection?}
\end{question}

%
%
%
%
%

\subsection*{Outline of the paper} In Section~2, we review \CAT geometry and $\kappa$ contracting geodesic rays. In section 3 we introduce the notion of $\kappa$-lower divergence prove Theorem $\ref{thm: first intro thm}.$ In section 4, we review the combinatorics and geometry of \CAT cube complexes, and prove Theorem \ref{theorem:counter example in the intro}, Theorem \ref{thm: 3rd theorem in the intro} and Corollary~\ref{cor:final statement of the intro}.

\subsection*{Acknowledgement.} 
The authors would like to thank Anthony Genevois, Mark Hagen, Merlin Incerti-Medici, Kasra Rafi, Michah Sageev and Giulio Tiozzo for helpful discussions.  We especially thank Kasra Rafi for suggesting the definition of $\kappa$-lower divergence and Anthony Genevois for fruitful comments on an earlier draft of the paper. The third author is very grateful to Mark Hagen for teaching him numerous cube complex techniques which were very helpful for the development of this paper. The authors would like to thank the anonymous referee(s).


\section{Preliminaries}

\subsection{\CAT geometry}
%
%
%
%
%
%

\begin{definition}[Quasi-geodesics] \label{Def:Quadi-Geodesic}

A \emph{geodesic ray} in $X$ is an isometric embedding $b : [0, \infty) \to X$. We fix a base-point $\go \in X$ and always assume 
that $b(0) = \go$, that is, a geodesic ray is always assumed to start from 
this fixed base-point. A \emph{quasi-geodesic ray} is a continuous quasi-isometric 
embedding $\beta : [0, \infty) \to X$ starting from $\go$. 
The additional assumption that quasi-geodesics are continuous is not necessary,
but it is added to make the proofs simpler. 

\end{definition}

\subsection{Basic properties of \CAT spaces}
A proper geodesic metric space $(X, d_{X})$ is \CAT if geodesic triangles in $X$ are at 
least as thin as triangles in Euclidean space with the same side lengths. To be precise, for any 
given geodesic triangle $\triangle pqr$, consider the unique triangle 
$\triangle \overline p \overline q \overline r$ in the Euclidean plane with the same side 
lengths. For any pair of points on the triangle, for instance on edges $[p,q]$ and $[p, r]$ of the 
triangle $\triangle pqr$, if we choose points $\overline x$ and $\overline y$  on 
edges $[\overline p, \overline q]$ and $[\overline p, \overline r]$ of 
the triangle $\triangle \overline p \overline q \overline r$ so that 
$d_X(p,x) = d_{\mathbb{E}^2}(\overline p, \overline x)$ and 
$d_X(p,y) = d_{\mathbb{E}^2}(\overline p, \overline y)$ then,
\[ 
d_{X} (x, y) \leq d_{\mathbb{E}^{2}}(\overline x, \overline y).
\] 

For the remainder of the paper, we assume $(X, d)$ is a proper \CAT space. A metric space $X$ is {\it proper} if closed metric balls are compact.
Here, we list some properties of proper \CAT spaces that are needed later (see 
 \cite{CAT(0)reference}). 

\begin{lemma} \label{Lem:CAT} 
 A proper \CAT space $X$ has the following properties:
\begin{enumerate}
\item It is uniquely geodesic, that is, for any two points $x, y$ in $X$, 
there exists exactly one geodesic connecting them. 

\item The nearest-point projection from a point $x$ to a geodesic line $b$ 
is a unique point denoted $x_b$. In fact, the closest-point projection map
\[
\pi_b : X \to b
\]
is Lipschitz. 
\item Convexity: For any convex set $Z \in X$, the distance function $f:X \rightarrow \mathbb{R}^+$ given by $f(x)=d(x,Z)$ is convex in the sense of \cite{BH1} II.2.1.
\end{enumerate}
\end{lemma}

\subsection{$\kappa$--contracting geodesics of $X$}
In this section we review the definition and properties of $\kappa$--contracting geodesic rays needed for this paper, for further details see \cite{QRT19}. We fix a function 
\[
\kappa : [0,\infty) \to [1,\infty)
\] 
that is monotone increasing, concave and sublinear, that is
\[
\lim_{t \to \infty} \frac{\kappa(t)} t = 0. 
\]
Note that using concavity, for any $a>1$, we have
\begin{equation} \label{Eq:Concave}
\kappa(a t) \leq a \left( \frac 1a \, \kappa (a t) + \left(1- \frac 1a\right) \kappa(0) \right) 
\leq a \, \kappa(t).
\end{equation}


The assumption that $\kappa$ is increasing and concave makes certain arguments
cleaner but they are not necessary. One can always replace any 
sublinear function $\kappa$, with another sublinear function $\overline \kappa$
so that $\kappa(t) \leq \overline \kappa(t) \leq \sC \, \kappa(t)$ for some constant $\sC$ 
and $\overline \kappa$ is monotone increasing and concave. For example, define 
\[
\overline \kappa(t) = \sup \Big\{ \lambda \kappa(u) + (1-\lambda) \kappa(v) \ST 
\ 0 \leq \lambda \leq 1, \ u,v>0, \ \text{and}\ \lambda u + (1-\lambda)v =t \Big\}.
\]
The requirement $\kappa(t) \geq 1$ is there to remove additive errors in the definition
of $\kappa$--contracting geodesics.  

\begin{definition}[$\kappa$--neighborhood]  \label{Def:Neighborhood} 
For a closed set $Z$ ( here $Z$ can be taken as  either a geodesic or a quasi-geodesic ray) and a constant $\nn$ define the $(\kappa, \nn)$--neighbourhood 
of $Z$ to be 
\[
\calN_\kappa(Z, \nn) = \Big\{ x \in X \ST 
  d_X(x, Z) \leq  \nn \cdot \kappa(x)  \Big\}.
\]

\end{definition} 

\begin{figure}[h!]
\begin{tikzpicture}
 \tikzstyle{vertex} =[circle,draw,fill=black,thick, inner sep=0pt,minimum size=.5 mm] 
[thick, 
    scale=1,
    vertex/.style={circle,draw,fill=black,thick,
                   inner sep=0pt,minimum size= .5 mm},
                  
      trans/.style={thick,->, shorten >=6pt,shorten <=6pt,>=stealth},
   ]

 \node[vertex] (a) at (0,0) {};
 \node at (-0.2,0) {$\go$};
 \node (b) at (10, 0) {};
 \node at (10.6, 0) {$b$};
 \node (c) at (6.7, 2) {};
 \node[vertex] (d) at (6.68,2) {};
 \node at (6.7, 2.4){$x$};
 \node[vertex] (e) at (6.68,0) {};
 \node at (6.7, -0.5){$x_{b}$};
 \draw [-,dashed](d)--(e);
 \draw [-,dashed](a)--(d);
 \draw [decorate,decoration={brace,amplitude=10pt},xshift=0pt,yshift=0pt]
  (6.7,2) -- (6.7,0)  node [black,midway,xshift=0pt,yshift=0pt] {};

 \node at (7.8, 1.2){$\nn \cdot \kappa(x)$};
 \node at (3.6, 0.7){$||x||$};
 \draw [thick, ->](a)--(b);
 \path[thick, densely dotted](0,0.5) edge [bend left=12] (c);
\node at (1.4, 1.9){$(\kappa, \nn)$--neighbourhood of $b$};
\end{tikzpicture}
\caption{A $\kappa$-neighbourhood of a geodesic ray $b$ with multiplicative constant $\nn$.}
\end{figure}
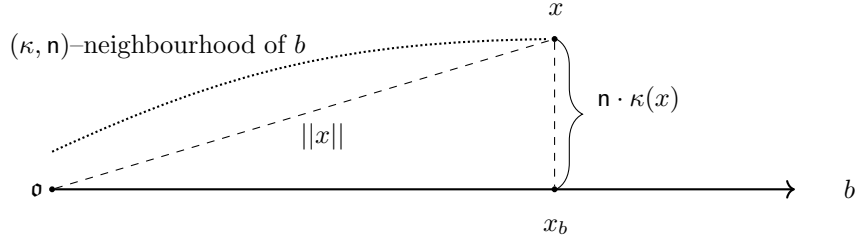


Lastly, for $x \in X$, define $\Norm{x} = d_X(\go, x)$. When used inside the function $\kappa(t)$, we drop $\Norm{\cdot}$ and write 
\[
\kappa(x) := \kappa(\Norm{x}). 
\]

\begin{definition}[$\kappa$--Morse] \label{Def:Morse} 
A geodesic ray $b$ is \emph{$\kappa$--Morse} if there is a function
\[
\mm_b : \RR_+^2 \to \RR_+
\]
so that if $\beta : [s,t] \to X$ is a $(\qq, \sQ)$--quasi-geodesic with end points 
on $b$ then
\[
\beta[s,t]  \subset \calN_{\kappa} \big(b,  \mm_b(\qq, \sQ)\big). 
\]
We refer to $\mm_{b}$ as the \emph{Morse gauge} for $b$.
\end{definition} 

\begin{definition}[$\kappa$--contracting] \label{Def:Contracting}
For a geodesic ray $b$ of $X$, we say $b$ is \emph{$\kappa$--contracting} if there 
is a constant $\cc_b$ so that, for every $x,y \in X$
\[
d_X(x, y) \leq d_X( x, b) \quad \Longrightarrow \quad
diam_X \big( x_b \cup y_b \big) \leq \cc_b \cdot \kappa( x).
\]
In the special case where $\kappa=1,$ we say that $b$ is \emph{uniformly contracting}.
\end{definition} 

In a CAT(0) space $X$, the notions of $\kappa$-Morse and $\kappa$-contracting are indeed equivalent.

\begin{theorem}[Theorem 3.8 \cite{QRT19}] \label{thm: ka Morse iff contracting}
Let $X$ be a CAT(0) space and let $b \in X$ be a geodesic ray. Then, b is $\kappa$-Morse if and only if it is $\kappa$-contracting.

\end{theorem}

Our  next  goal  is  to  prove  that  the  notion  of $\kappa$-contracting is equivalent to another notion of contraction. In  order  to  do  so,  we  need  the  following  two propositions (see \cite{ChSu2014}, \cite{QRT19}).

\begin{proposition} [Proposition 3.11 \cite{QRT19}]\label{Proposition: Constructing a close quasi-geodesic} Given a geodesic segment $b$ (possibly infinite) and points $x,y \in X$ with $d(x,y)<d(x,b)$, there exists a $(32, 0)$-quasi-geodesic $\beta:[s_0,s_1] \rightarrow X$ with end points on $b$ such that $\beta (s_0)=x_b$, \begin{center}
    $\frac{1}{4}d(x_b,y_b)\leq d(\beta(s_0), \beta(s_1))<d(x_b,y_b)$
\end{center}
and there a point $p=\beta(t)$ on $\beta$ such that 
\begin{center}
    $d(p,b) \geq \frac{1}{80}d(x_b,y_b).$
\end{center}

\end{proposition}
The following proposition states that if $x,y$ are within $\kappa(x)$ of each other, then $\kappa(x)$ and $\kappa(y)$ are multiplicatively the same.

\begin{proposition}[{\cite[Lemma 3.2]{QRT19}}] \label{lemma: constants}
For any $\dd_0>0$, there exists $\dd_1,\dd_2>0$ depending only on $\dd_0$ and $\kappa$ so that for $x,y \in X$, we have 
\begin{center}
    $d(x,y) \leq \dd_0 \kappa(x) \implies \dd_1\kappa(x) \leq \kappa(y) \leq \dd_2 \kappa(x).$
\end{center}
\end{proposition}

Now we present a slightly different but equivalent contracting condition that we will frequently use in the paper:

\begin{lemma} \label{lem: contracting implies projection contracting}
A geodesic ray is $\kappa$--contracting if and only if there exists a constant $\cc>0$ such that for any $x,y \in X$ with $d(x,y)<d(x,b)$, we have 
\[ d(x_b,y_b) \leq \cc \kappa(x_b). \]

\end{lemma}

\begin{proof}
($\Rightarrow$) Since $d(x, y) < d(x, b)$, there exists a quasi-geodesic $\beta$ as in Proposition~\ref{Proposition: Constructing a close quasi-geodesic}, and a point $p \in \beta$ such that 

\[
d(x_{b}, y_{b}) \leq 80d(p, b).\]
Using Theorem \ref{thm: ka Morse iff contracting}, since $b$ is $\kappa$--Morse,
\[ d(p, b) =d(p, p_{b}) \leq \mm_b(32, 0) \kappa(p) \]
Applying Proposition~\ref{lemma: constants} to points $p, p_{b}$, we have

\begin{equation}\label{bound1}
d(x_{b}, y_{b}) \leq 80 d(p, p_{b}) \leq 80 \mm_b(32, 0) \kappa(p) \leq 80 \mm_b(32, 0)\kappa(p_b)/\dd_{1}. 
\end{equation}

Consider the distance between $x_{b}$ and $p_{b}$:
\begin{align*}
d(x_{b}, p_{b})  &\leq d(x_{b}, p) \quad \text{projections in \CAT spaces are Lipschitz.} \\
                         & \leq 32 d(x_{b}, y_{b}) + 0 \quad \text{Proposition~\ref{Proposition: Constructing a close quasi-geodesic}} \\
                         & \leq 32 \cdot 80 \mm_b(32, 0)\kappa(p_b)/\dd_{1} \quad \text{Equation~\ref{bound1}} \\
\end{align*}

Thus applying Proposition~\ref{lemma: constants} again, we get some constant $\dd_{2}$ such that $\kappa(p_b) \leq \dd_{2} \kappa(x_{b})$.  Combined with Equation~\ref{bound1}, there exists $\dd_{3}$ such that

\[ d(x_{b}, y_{b}) \leq  \dd_{3} \kappa(x_{b}). \]

($\Leftarrow$) Since projections are distance-decreasing in CAT(0) spaces, 
 \[ d(x_b,y_b) \leq \cc \kappa(x_b) \leq \cc \kappa(x), \]
 as desired.

\end{proof}

\section{$\kappa$-Lower Divergence and slim geodesics}\label{sec:divergence}
In this section we study the divergence and the slimness property of $\kappa$-contracting geodesics. We first give the definitions of these measures. Divergence measures how fast the circumference of a ball grows as its radius increases. Under the lower divergence notion introduced by Charney and Sultan in \cite{ChSu2014}, $\kappa$-contracting rays cannot be distinguished from geodesic rays which lie in a flat. We introduce a lower divergence notion which is more sensitive.
\subsection{$\kappa$-lower divergence of geodesics}
\begin{definition}[$\kappa$-lower divergence] 
Let $b$ be a geodesic ray and fix some  $r>0$ and then $t > r \kappa(t)$. Let $\rho_\kappa (r,t)$ denote the infimum of the lengths of all paths from $b(t-r\kappa(t))$ to $b(t+r\kappa(t))$ which lie outside the open ball of radius $r\kappa(t)$ about $b(t)$.    If there does not exist such a path, then let $\rho_\kappa (r,t) = \infty$. Given a geodesic ray $b$, we define the $\kappa$-\emph{lower divergence} of $b$ to be the function $$div_{\kappa}(r):= \underset{t>r\kappa(t)}{\text{inf}}\,\,\frac{\rho_\kappa(r,t)}{\kappa(t)}$$

\end{definition}
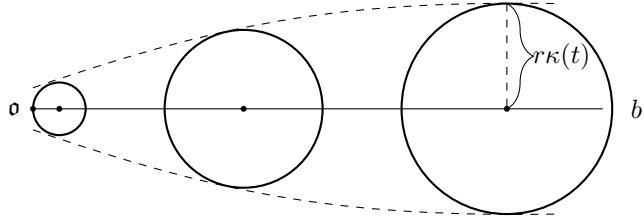
\begin{figure}[H]


\begin{tikzpicture}[scale=0.7]
 \tikzstyle{vertex} =[circle,draw,fill=black,thick, inner sep=0pt,minimum size=.5 mm]
[thick, 
    scale=1,
    vertex/.style={circle,draw,fill=black,thick,
                   inner sep=0pt,minimum size= .5 mm},
                  
      trans/.style={thick,->, shorten >=6pt,shorten <=6pt,>=stealth},
   ]

  \node[vertex] (o) at (0,0)[label=left:$\go$] {}; 
  \node(a) at (11,0)[label=right:$b$] {}; 
  
  \draw (o)--(a){};

  \draw [dashed] (0, 0.4) to [bend left = 10] (10,2){};
  
    \draw [dashed] (0, -0.4) to [bend right = 10] (10,-2){};
\draw [thick] (0.5,0) circle (0.5cm);
\draw [thick] (4,0) circle (1.5cm);
\draw [thick] (9,0) circle (2cm);
\node[vertex] at (0.5, 0){}; 
\node[vertex] at (4, 0){}; 
\node[vertex] at (9, 0){};

  \draw [decorate,decoration={brace,amplitude=10pt},xshift=0pt,yshift=0pt]
  (9, 2) -- (9,0)  node [thick, black,midway,xshift=0pt,yshift=0pt] {};       
  \draw [dashed]  (9, 2) -- (9,0){};
 \node at (10,1) {$\, r\kappa(t)$};

  \end{tikzpicture}
  
\caption{Definition of $\kappa$-lower divergence.}
\label{fig: definition of lower divergence}
\end{figure}

The specific $\kappa$-lower-divergence function is rarely of direct interest, instead the growth rate of the function is the primary property of interest.  Two functions $f(t)$ and $g(t)$ have the same \emph{growth rate} if the limit of their ratios is non-zero and finite,
\[
0 <  \lim_{t \to \infty} \left|\frac{f(t)}{g(t)}\right| < \infty
\]
We say that the $\kappa$-lower-divergence is \emph{linear} if it's growth rate is the same as the growth rate of $t$, and it is \emph{superlinear} if $\displaystyle \lim_{t \to  \infty}\frac{div_\kappa(t)}{t} = \infty$. Similarly we say that the $\kappa$-lower-divergence is \emph{quadratic} if the growth rate is the same as $t^2$, and we say it is exponential if there is some   $a>1$ so that it has the same growth rate as $a^t$.

%
%
%
%
%

 To proceed we discuss slim triangles. The idea of capturing coarse hyperbolicity by slim triangles first came by Gromov in \cite{Gro87}. Here we give four different adaptations of the concept and we shall see in Corollary~\ref{310} that they are equivalent in the setting of CAT(0) spaces.

\begin{definition}\label{different but equivalent 3 def} Let $b$ be an infinite geodesic ray in a CAT(0) space. We  describe four different conditions on the slimness of the geodesic ray:
\begin{enumerate}
    \item We say that $b$ satisfies the \emph{$\kappa$-slim condition 1} if there exists some $\cc \geq 0$ such that for any $x \in X, y \in im(b)$, we have $d(x_b, [x,y]) \leq \cc\kappa(x_b)$.
    
    \item We say that $b$ satisfies the \emph{$\kappa$-slim condition 2} if there exists some $\cc \geq 0$ such that for any $x \in X, y \in im(b)$, we have $d(x_b, [x,y]) \leq \cc\kappa(z)$ for some $z$ between $x_b$ and $y.$

    \item We say that $b$ satisfies the \emph{ $\kappa$-slim condition 3} if there exists some $\cc \geq 0$ such that for any $x,y,z \in X$ with   $[y,z] \subset im(b)$, if $w \in [y,z]$, we have $d(w,[y,x] \cup [x,z]) \leq \cc \kappa(x_b)$.
    
    \item We say that $b$ satisfies the \emph{$\kappa$-slim condition 4} if there exists some $\cc \geq 0$ such that for any $x \in X$,  $[y,z] \subset im(b)$ with $||y|| \leq ||x_b|| \leq ||z||$ if $w \in [y,z],$ we have $d(w,[y,x] \cup [x,z]) \leq \cc \kappa(z).$

\end{enumerate}
\end{definition}

We collect now the implications between these definitions:
\begin{lemma}\label{lem:imply}
Let $b$ be an infinite geodesic ray in a CAT(0) space.
\begin{enumerate}[a).]
\item $\kappa$-slim condition 1 $\Longrightarrow \kappa$-slim condition 2.
\item $\kappa$-slim condition 3 $\Longrightarrow \kappa$-slim condition 4.
\item $\kappa$-slim condition 1 $\iff \kappa$-slim condition 3.
\item $\kappa$-slim condition 4 $\Longrightarrow \kappa$-slim condition 2.
\end{enumerate}

\end{lemma}


\begin{proof}
Let $b$ be an infinite geodesic ray in a CAT(0) space in each statement:
\begin{enumerate}[a).]
\item Choosing $z=x_b$ gives the desired statement.
\item Since $||x_b|| \leq ||z||$, we have

\begin{center}
    $d(w,[y,x] \cup [x,z]) \leq \cc \kappa(x_b) \implies d(w,[y,x] \cup [x,z]) \leq \cc \kappa(z)$.
\end{center}

\item 

$\kappa$-slim condition 1 $\Rightarrow \kappa$-slim condition 3: Let $x,y,z \in X$ with $[y,z] \in im(b)$, using $\kappa$-slim condition 1, we have $d(x_b,[y,x]) \leq \cc \kappa(x_b),$ convexity of the CAT(0) metric then implies that $d(w,[y,x]) \leq \cc \kappa(x_b)$ for all $w \in [y,x_b].$ Similarly, using $\kappa$-slim condition 1, we get that $d(x_b, [x,z]) \leq \cc \kappa(x_b)$. Again, by convexity of the CAT(0) metric, we get that $d(w,[x,z] \leq \cc \kappa(x_b))$.

\noindent $\kappa$-slim condition 3 $\Rightarrow \kappa$-slim condition 1:  Let $x \in X, y \in im(b)$. If $d(x_b,y) \leq \cc \kappa(x_b),$ then we have $d(x_b,[x,y]) \leq d(x_b,y)  \leq \cc \kappa(x_b)$. Otherwise, if $d(x_b,y)> \cc \kappa(x_b),$ we can choose $x'$ between $x_b$ and $y$ with $\cc\kappa(x_b)<d(x',x_b) \leq 2\cc \kappa(x_b).$ Using the $\kappa$-slim condition 3, we get $d(x',[x,y]) \leq \cc \kappa(x_b).$ Therefore, by the triangle inequality, we have $d(x_b,[x,y])\leq 3\cc \kappa(x_b).$

\item
Let $x \in X, y \in im(b)$ and define $z=\text{max}\{x_b,y\}$. If $d(x_b,y) \leq \cc \kappa(z),$ we have $d(x_b,[x,y]) \leq d(x_b,y) \leq \cc \kappa(z).$
 Otherwise, we can choose $x'$ between $x_b$ and $y$ with $\cc\kappa(z)<d(x',x_b) \leq 2\cc \kappa(z)$. Using the $\kappa$-slim condition 4, we get $d(x',[x,y]) \leq \cc \kappa(z).$ Therefore, by the triangle inequality, we have $d(x_b,[x,y])\leq 3\cc \kappa(z).$
\end{enumerate}
\end{proof}

\begin{lemma}\label{travelling in sublinear nbhds}
Let $b$ be a $\kappa$-contracting geodesic ray with contracting constant $\cc$. For any $x \in X$ and $y \in b$ we have

 $$d(x, x_b)+d(x_b, y)- \cc\kappa(||x_b||)-1  \leq d(x,y) \leq d(x,x_b)+d(x_b, y).$$
                  
\end{lemma}

\begin{figure}[!h]
\begin{tikzpicture}[scale=0.7]
\draw[thin,->] (-5,0) -- (6,0);

\node[below] at (-5,0) {$\go$};

\node[left] at (0,3) {$x$};

\draw[thick,fill=black] (5,0) circle (0.05cm);

\node[above] at (5,0) {$y$};

\draw[thick,fill=black] (0,3) circle (0.05cm);

\draw[thick,fill=black] (0,0) circle (0.05cm);

\node[below] at (0,0) {$x_b$};

\draw[thick,fill=black] (1.2,.65) circle (0.05cm);

\node[above] at (1.2,.65) {$z$};
\draw[thick,fill=black] (1.2,0) circle (0.05cm);

\node[below] at (1.2,0) {$z_b$};

\draw[very thick,black] (0,3) .. controls ++(0,-3) and ++ (-2,0)..(5,0);

\end{tikzpicture}

\caption{ The distance from $x$ to $y$ is approximated by taking the projection to $b$ and then going along $b$.}\label{approximating distance}

\end{figure}
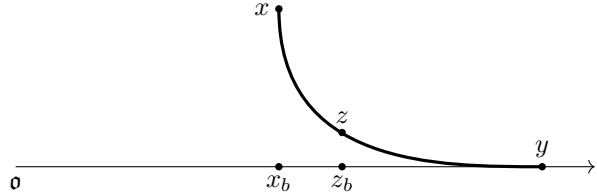

\begin{proof}
Consider a ball around $x$ of radius $r=d(x,x_b)-1$. Let $z$ be the point on $[x,y]$ with $d(x,z)=r$. Since $b$ is $\kappa$-contracting, we have that \[d(x_b,z_b) \leq \cc\kappa(x_b).\]
Notice that

\begin{align*}
d(x,y)-d(x, x_b) &=d(x,z)+d(z,y)-d(x,x_b)\\
                                          &=(d(x,x_b)-1)+d(z,y)-d(x,x_b)\\
                                          &=d(z,y)-1 \\
                                          &\geq d(z_b, y)-1\\
                                          & \geq d(x_b,y)-d(x_b, z_b)-1 \\
                                          &\geq d(x_b, y)- \cc \kappa(x_b)-1.
\end{align*} 
This yields

\[ d(x,y) \geq d(x, x_b)+d(x_b, y)-\cc\kappa(x_b)-1,\]
where the second inequality follows from triangle inequality.

\end{proof}

\begin{lemma}\label{contracting implies slim}
Every $\kappa$-contracting geodesic ray is $\kappa$-slim condition 1.
\end{lemma}

\begin{proof}

Let $x \in X$ and let $y \in b$. Let $z'$ be the point on $[x,y]$ with $d(x,z')=d(x,b).$ Since $b$ is $\kappa$-contracting, using Lemma \ref{lem: contracting implies projection contracting}, we have $d(x_b, z'_b) \leq \cc \kappa(x_b).$ Now, as \[d(y, x_b)+d(x_b, x) \geq d(x,y)=d(x,z')+d(z',y),\] we get that $d(z',y) \leq d(y,x_b).$ Thus, we have 

\begin{equation}\label{E1}
d(z',y) \leq d(y, x_b) \leq d(y, z'_b)+\cc\kappa (x_b).
\end{equation}
 Now, applying Lemma~\ref{travelling in sublinear nbhds} to the points $z'$ and $y$, we get that 

\begin{equation}\label{E2}
d(z',y) \geq d(z', z'_b)+d(z'_b, y)-\cc\kappa(z'_b)-1.
\end{equation}

Since $d(x_b,z'_b) \leq \cc \kappa(x_b)$, we have $||z'_b|| \leq ||x_b||,$ or \[||z'_b||\leq ||x_b||+\cc \kappa(x_b) \leq ||x_b||+\cc||x_b||=(1+\cc)||x_b||.\]
 Combining the inequalities \ref{E1} and \ref{E2}, we get: 

\[ d(z', z'_b)) \leq \cc \kappa(x_b)+ \cc \kappa(z'_b)+1 \leq  \cc \kappa(x_b)+ \cc \kappa((1+\cc)x_b)+1.\] Therefore

  \[ d(z', z'_b) \leq (\cc+\cc(\cc+1))\kappa(x_b)+1 \] Thus

\[d(z', x_b) \leq d(z', z'_b)+d(z'_b, x_b) \leq (\cc+\cc(\cc+1))\kappa(x_b)+1+\cc \kappa(x_b). \] Which gives 

\[d(z', x_b) \leq (3\cc+\cc^2)\kappa(x_b)+1 \leq (3\cc+\cc^2)\kappa(x_b)+\kappa(x_b) .\]

Thus,
\[d(z', x_b) \leq (3\cc+\cc^2+1)\kappa(x_b) .\] Choosing $\cc'=3\cc+\cc^2+1$ gives the desired result.

\end{proof}

\begin{proposition}\label{thin iff contracting} A geodesic ray $b$ is $\kappa$-contracting if and only if it satisfies $\kappa$-slim condition 2.

\end{proposition}

$$
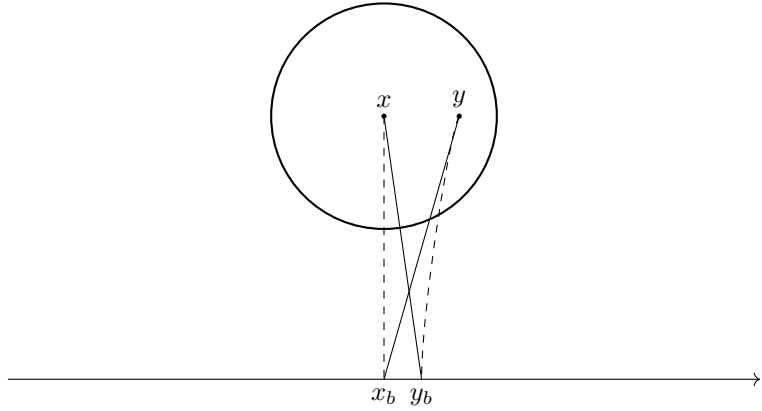
\begin{tikzpicture}
\draw[thin,->] (-5,0) -- (5,0);

\draw[thick] (0,3.5) circle (1.5cm);
\node[above] at (0,3.5) {$x$};
\node[below] at (0,0) {$x_b$};
\draw[thin] (0,3.5) -- (.5,0);
\draw[thin ] (1,3.5) -- (0,0);

\draw[thin,dashed] (0,3.5) -- (0,0);

\draw[dashed] (1,3.5) node[above] {$y$} .. controls ++(-.1,-.1) and ++(0,1) .. (.5,0) node[below] {$y_b$};

\draw[thick,fill=black] (0,3.5) circle (0.02cm);
\draw[thick,fill=black] (1,3.5) circle (0.02cm);

\end{tikzpicture}
$$

\begin{proof}  We have already shown that $\kappa$-contracting geodesic rays are $\kappa$-slim condition 1, and by Lemma \ref{lem:imply} a) it also satisfies $\kappa$-slim condition 2.  Suppose that $b$ is $\kappa$-slim condition 2, let $x \in X$ and let $B$ be a disjoint ball around $x$ of radius $r.$ Let $y \in B=B(x,r)$ and let $x_b, y_b$ be the respective projections of $x,y$ on $b$. Since $b$ is $\kappa$-slim condition 2, there exists a point $w \in [x_b,y]$ with $d(y_b,w) \leq \cc \kappa(z)$ for some $z \in [x_b,y_b].$ Using convexity of the CAT(0) metric, since $d(x,y) \leq d(x,x_b)$, we have $ d(x,w) \leq max \{d(x,x_b),d(x,y) \}=d(x,x_b).$

Therefore

\begin{center}
$d(x,y_b) \leq d(x,w)+\cc\kappa(z) \leq d(x,x_b)+\cc\kappa(z).$
\end{center}

On the other hand, applying the $\kappa$-slim condition 2 to $\Delta=\Delta (x,x_b,y_b)$, we get a point $w'$ on $[x,y_b]$ with $d(x_b, w') \leq \cc\kappa(z')$ for some $z' \in [x_b,y_b]$. Hence, we have
 
 \begin{align*}
d(x,y_b)&=d(x,w')+d(w',y_b)\\
&\geq [d(x,x_b)-\cc \kappa(z')] +[d(x_b,y_b)-\cc \kappa(z')]\\
&=d(x,x_b)+d(x_b,y_b)-2\cc\kappa(z'). 
\end{align*}

This implies

\begin{center}
    
    $d(x,x_b)+d(x_b,y_b)-2\cc\kappa(z') \leq d(x,y_b) 
    \leq d(x,x_b)+ \cc \kappa(z).$
    \end{center}

 Therefore, we have 
 
 \begin{center}
     $d(x_b,y_b) \leq 2\cc \kappa(z')+\cc\kappa(z).$
 \end{center}

We have two cases, if $x_b \leq y_b$, then

\begin{align*}
    d(x_b,y_b) &\leq 2\cc \kappa(z')+\cc\kappa(z)\\
    &\leq 2\cc \kappa(y_b)+\cc\kappa(y_b)\\ &\leq 2\cc \kappa(y)+2\cc\kappa(y)\\
    &=4\cc \kappa(y)\\
    &\leq 4 \cc \kappa(2x)\\ 
    &\leq 8 \cc \kappa(x).
\end{align*}

Now, if $y_b \leq x_b$, then 

\begin{align*}
d(x_b,y_b) &\leq 2\cc \kappa(z')+\cc\kappa(z)\\
&\leq 2\cc \kappa(x_b)+\cc\kappa(x_b)\\
&\leq 2\cc \kappa(x)+2\cc\kappa(x)\\
&=4\cc \kappa(x).
\end{align*}

\end{proof}

\begin{corollary}\label{310}
All the $\kappa$-slim conditions are equivalent.
\end{corollary}

\begin{proof}
Using Theorem \ref{thin iff contracting} and Lemma \ref{contracting implies slim}, we get that $\kappa$-slim condition 2 implies $\kappa$-slim condition 1. Now, applying Lemmas \ref{lem:imply} we get the desired result.

\end{proof}

Therefore we give the following definition.
\begin{definition}[$\kappa$-slim geodesic]
A geodesic ray $b$ is said to be \emph{$\kappa$-slim} if it satisfies any of the four $\kappa$-slim conditions.
\end{definition}

\begin{theorem} The $\kappa$-lower divergence for a $\kappa$-contracting geodesic ray $b$ is at least quadratic.

\end{theorem}
$$
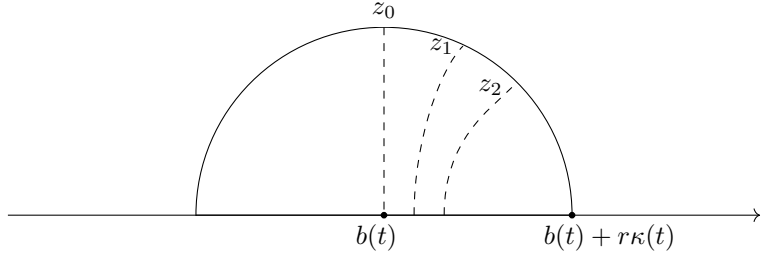
\begin{tikzpicture}
\draw[thin,->] (-5,0) -- (5,0);

\node[right] at (2,-.3) {$b(t)+r\kappa(t)$};
\draw[thick,fill=black] (2.5,0) circle (0.03cm);
\draw[thick,fill=black] (0,0) circle (0.03cm);

\draw[dashed] (.4,0) .. controls ++(0,1) and ++(-.3,-.3) .. (1.05,2.25) node[left] {$z_1$};
\draw[dashed] (.8,0) .. controls ++(0,1) and ++(-.3,-.3) .. (1.7,1.7) node[left] {$z_2$};

\draw[thin, dashed] (0,2.5) -- (0,0);
\node[above] at (0,2.5) {$z_0$};
\node[below] at (-.1,0) {$b(t)$};

\begin{scope}
    \clip (-2.5,0) rectangle (2.5,2.5);
    \draw (0,0) circle(2.5);
    \draw (-2.5,0) -- (2.5,0);
\end{scope}

\end{tikzpicture}
$$

\begin{proof}

Suppose that $b$ is $\kappa$-contracting. Consider a ball $B=B_{r \kappa(t)}(b(t))$ with $t>r\kappa(t)$. Let $\beta$ be a path avoiding $B$ connecting $b(t-r\kappa(t))$ to $b(t+r \kappa(t))$. Choose $z_0 \in im(\beta)$ with $\pi_b(z_0)=b(t).$ Since $d(z_0,b(t)) \geq r\kappa(t),$ we have \[d(z_0,b(t+r\kappa(t))>r\kappa(t),\] and hence we may choose $z_1 \in im(\beta)$ with $d(z_0,z_1)=r \kappa(t)$. Now, if $x_1= \pi_b(z_1)$, then $d(b(t),x_1) \leq \cc\kappa(2t).$ By triangle inequality, \[d(z_1,x_1) \geq r\kappa(t)-\cc\kappa(2t)\] and hence \[d(z_1, b(t+r\kappa(t))> r\kappa(t)-\cc\kappa(2t).\] Thus, we may choose $z_2 \in im (\beta)$ with $d(z_2,z_1)=r \kappa(t)-\cc \kappa(2t).$ We iterate this process to get a sequence $z_0,z_1,z_2,...,z_n$ on $\beta$, with respective projections $b(t)=x_0, x_1,...,x_n$ such that 
\[d(z_i,z_{i+1})=r\kappa(t)-i\cc \kappa(2t) \quad \text{ and } \quad d(x_i,x_{i+1}) \leq \cc\kappa(2t),\] where $n$ is an integer satisfying $\frac{r\kappa(t)}{\cc\kappa(2t)}-1 \leq n \leq \frac{r\kappa(t)}{\cc\kappa(2t)}$. We remark that since $\kappa$ is an increasing function and since $\kappa(2t) \leq 2 \kappa(t)$, we have $1 \leq \frac{\kappa(2t)}{\kappa(t)}  \leq 2$. We also have
     \begin{align*}
        |\beta| &\geq d(z_0,z_1)+...+d(z_{n-1}, z_n) \\
                &= (r\kappa(t))+(r\kappa(t)-\cc\kappa(2t))+(r\kappa(t)-2\cc\kappa(2t))+...+(r\kappa(t)-(n-1)\cc\kappa(2t))\\
                &=nr\kappa(t)-\frac{n(n-1)\cc\kappa(2t)}{2} \\
                &\geq nr \kappa(t)-\frac{n^2\cc\kappa(2t)}{2}.
     \end{align*}

Thus

\begin{align*}
\frac{|\beta|}{r \kappa(t)} & \geq n-\frac{n^2 \cc \kappa(2t)}{2r \kappa(t)}\\
                            &\geq \frac{r\kappa(t)}{\cc\kappa(2t)}-1-\big( \frac{r\kappa(t)}{c\kappa(2t)} \big)^{2}\frac{\cc \kappa(2t)}{2r\kappa(t)}\\
                            &=\frac{\kappa(t)}{\cc\kappa(2t)}(r-\frac{r}{2})-1\\
                            &=\frac{\kappa(t)r}{2\cc\kappa(2t)}-1 \\
                            &\geq \frac{r \kappa(t)}{2\cc 2 \kappa(t)}-1\\ &=\frac{r}{4\cc}-1. 
                            \end{align*}
Hence, we have $\frac{|\beta|}{\kappa(t)} \geq \frac{r^2}{4c}-r$. Since $\beta$ was an arbitrary path avoiding $B$, we conclude

\begin{center}

$div_{\kappa}(r)= \underset{t>r\kappa(t)}{\text{inf}}\,\,\frac{\rho_\kappa(r,t)}{\kappa(t)} \geq \frac{r^2}{4\cc}-r$.

\end{center}

\end{proof}


\begin{theorem} If the $\kappa$-lower divergence of a geodesic ray $b$ is superlinear, then $b$ is $\kappa$-slim.
\end{theorem}

\begin{figure}[h!]

\begin{tikzpicture}[scale=0.6]
 \tikzstyle{vertex} =[circle,draw,fill=black,thick, inner sep=0pt,minimum size=.5 mm]
[thick, 
    scale=1,
    vertex/.style={circle,draw,fill=black,thick,
                   inner sep=0pt,minimum size= .5 mm},
                  
      trans/.style={thick,->, shorten >=6pt,shorten <=6pt,>=stealth},
   ]

 \draw[thick]  (-10, 0)--(10, 0){};
  \draw[thick]  (0, 0)--(0, 9){};
    \draw[thick, dashed]  (0, 0)--(2.15, 2.15){};
 \draw [thick] (3,0) arc (0:180:3cm);
  \draw[thick, dashed]  (6,0) arc (0:180:6cm);
  \draw[thick]  (-9,0) .. controls (-1.75,1) and (-0.75, 2) .. (0,9);
    \draw [thick] (9,0) .. controls (1.75, 1) and (0.75,2) .. (0,9);

  \node[vertex] at (-9,0)[label=below:$y_{r}$] {}; 
  \node[vertex] at (9,0)[label=below:$z_{r}$] {}; 
   \node[vertex] at (0,9)[label=right:$x_{r}$] {};

  \node[vertex] at (-6,0)[label=below:$a$] {}; 
   \node[vertex] at (-7,0.3)[label=above:$a'$] {}; 
 \node[vertex] at (6,0)[label=below:$c$] {}; 
  \node[vertex] at (7,0.3)[label=above:$c'$] {}; 
  \node[vertex] at (-3,0)[label=below:$p_{1}$] {}; 
    \node[vertex] at (3,0)[label=below:$p_{2}$] {}; 
      \node[vertex] at (0,0)[label=below:$t_{r}$] {}; 
       \node[vertex] at (0,6)[label=below right :$e$] {}; 
       \node[vertex] at (-0.25,7)[label=left:$e'$] {}; 
       \node[vertex] at (0.25,7)[label=right:$e''$] {};

 
  \draw [decorate,decoration={brace,amplitude=10pt},xshift=0pt,yshift=0pt]
  (0,-0.8) -- (-6,-0.8)  node [thick, black,midway,xshift=0pt,yshift=0pt] {};       
 \node at (-3,-2) {$3r \kappa(t_{r})$};
  \node at (1.6,0.7) {$r \kappa(t_{r})$};

  \end{tikzpicture}

\end{figure}

\begin{proof}

Suppose that $b$ is not $\kappa$-slim. In particular, $b$ is not $\kappa$-slim with respect to condition 4. Thus, we have that for any $r >0$, there exist three points $x_r, y_r, z_r \in b$ as in condition 4 with \[||y_r|| \leq ||(x_r)_b||\leq ||z_r||\] and also a point $t_r \in [y_r,z_r]$ such that 

\begin{center}
    
$d(t_r,[y_r,x_r] \cup [z_r,x_r]) \geq r \kappa(z_r) \geq r \kappa(t_r).$

\end{center}

Note that this implies in particular that 
\[
d(y_r, t_r) \geq r \kappa(t_r) \text{ and }
d(z_r, t_r) \geq r \kappa(t_r)
\]

Up to replacing $x_r$ by some $x_r'$ in $[x_r,(x_r)_b],$ we may assume \[d(t_r,[y_r,x_r] \cup [z_r,x_r])=r \kappa(t_r).\] Thus, we have $d(t_r,[y_r,x_r])=r\kappa(t_r)$ or $d(t_r,[z_r,x_r])=r\kappa(t_r)$, without loss of generality, say $d(t_r,[y_r,x_r])=r\kappa(t_r)$. Up to replacing $[x_r,z_r]$ by $[x_r,z_r']$ with $||z_r'|| \leq ||z_r||,$ we may also assume $d(t_r,[x_r,z_r])=r \kappa(t_r)$. Therefore, we have $x_r,y_r,z_r,t_r$ with \begin{center}

$d(t_r,[y_r,x_r])=r\kappa(t_r)$ and $d(t_r,[z_r,x_r])=r\kappa(t_r).$
\end{center}

 Let $a$ be the point on $[y_r,t_r]$ at distance $3r\kappa(t_r)$ from $t_r$ (if no such $a$ exists, choose $a=y_r)$ and let $a'$ be the projection of $a$ on $[y_r,x_r].$ By convexity of the CAT(0) metric, we have \[ d(a,a')=d(a,[y_r,x_r]) \leq d(t_r,[y_r,x_r])=r\kappa(t_r).\] 
 
 Similarly, if $c$ is the point in $[t_r, z_r]$ at distance $3r \kappa(t_r)$ from $t_r$ (or $c=z_r$ if $d(t_r,z_r)< 3 \kappa(t_r))$ and $c'$ is the projection of $c$ on $[x_r,z_r],$ then we have \[d(c,c') \leq r \kappa(t_r).\] Now, let $e$ be the point on $[x_r,t_r]$ at distance $3r \kappa(t_r)$ from $t_r$ (or $e=x_r$ if $d(t_r,x_r)< 3 \kappa(t_r)$)  and let $e', e''$ be the projections of $e$ on $[x_r, y_{r}],[x_r,z_r]$ respectively. By the convexity of the CAT(0) metric, we have 
  \[d(e,e') \leq r\kappa(t_r) \quad \text{ and } \quad d(e,e'')  \leq r\kappa(t_r) .\]  Also, since projections in CAT(0) spaces are distance decreasing, we have 
  \[d(a',e') \leq 6r \kappa(t_r) \quad \text{ and } \quad d(c',e'') \leq 6r \kappa(t_r).\] Now, consider the path $\beta$ connecting $p_1=b(t_r- r \kappa(t_r))$ to $p_2=b(t_r+r \kappa(t_r))$
 \begin{center}
     
 $\beta=[p_1,a] \cup [a,a'] \cup [a',e'] \cup [e',e] \cup [e, e'']\cup [e'',c'] \cup[c',c] \cup [c,p_2]$.
  \end{center}
  The length of this path \[|\beta| \leq 2r \kappa(t_r)+r \kappa(t_r)+6r \kappa(t_r)+r\kappa(t_r)+r\kappa(t_r)+6r\kappa(t_r)+r\kappa(t_r)+2r\kappa(t_r)=20r \kappa(t_r).\]
  
   Therefore, we have $\frac{|\beta|}{\kappa(t_r)} \leq 20r$. It remains to show that $\beta$ lies outside the ball $B=B(t_r,r\kappa(t_r))$. Notice that we either have $d(a, t_r)=3r \kappa(t_r)$, where the distance $d(a,B(t_r, r\kappa(t_r))=2r \kappa(t_r)>r \kappa(t_r)$ and hence the geodesic $[a,a']$ lies outside $B=B(t_r,r\kappa(t_r))$ as $d(a,a') \leq r\kappa(t_r)$; or we have $a= y_r$ and $a$ is outside $B=B(t_r,r\kappa(t_r))$ by the convexity of the CAT(0) metric. Similarly, the geodesics $[c,c'], [e,e']$ and $[e,e'']$ all lie outside $B$. On the other hand, the geodesics $[p_1,a], [a',e'],[e'',c']$ and $[c,p_2]$ also lie outside $B$ by construction.

\end{proof}

\section{$\kappa$--contracting geodesic rays in CAT(0) cube complexes}
The main theorem we aim to prove in this section is Theorem~\ref{thm: 3rd theorem in the intro}, restated here:
\begin{theorem} \label{thm: main theorem section 4}
Let $X$ be a locally finite cube complex. A geodesic ray $b \in X$ is $\kappa$-contracting if and only if there exists $\cc>0$ such that $b$ crosses an infinite sequence of hyperplanes $h_1, h_2,...$ at points $b(t_i)$ satisfying: 

\begin{enumerate}
    \item $d(t_i,t_{i+1}) \leq \cc \kappa(t_{i+1}).$
    \item $h_i,h_{i+1}$ are $\cc \kappa(t_{i+1})$-well-separated.
\end{enumerate}
\end{theorem}

First we refer to geodesic ray that satisfies the exact condition of this theorem as  \emph{$\kappa$--excursion geodesics}. The constant $\cc$ will be referred to as the \emph{excursion constant}.

\begin{definition}[\CAT cube complexes]
Given $d \geq 0$, a $d$--cube is a copy of $[0,1]^{d}$, equipped
 with Euclidean metric. Its dimension is $d$. A mid-cube is a subspace obtained by restricting exactly one coordinate to   1/2. Mid-cubes are $(d-1)$--cubes.

A \emph{cube complex} is a CW complex $X$ such that 
\begin{itemize}
\item each cell is a $d$--cube for some $d$;
\item each cell is attached using an isometry of some face.
\end{itemize}
The dimension of $X$ is the supremum of the set of dimensions of cubes of $X$. A \emph{\CAT cube complex} is a cube complex whose underlying space is a \CAT space. We put an equivalence relation on the set of
mid-cubes generated by the condition that two mid-cubes are equivalent if they
share a face. A \emph{hyperplane} is the union of all the mid-cubes in an equivalence
class. 

\end{definition}

Let $X$ be a finite dimensional CAT(0) cube complex (not necessarily uniformly locally finite) of dimension $v$. Let $X^{(1)}$ denote the 1-skeleton of $X$.  We denote the distance function in the CAT(0) space $X$ by $d$, and the $\ell^{1}$-distance by $d^{(1)}.$ A geodesic in the $d$-metric is said to be a \emph{CAT(0) geodesic} while a geodesic in the $d^{(1)}$-metric is said to be a \emph{combinatorial geodesic}. For $x,y \in X^{(1)}$, we let  $\mathcal{C}_{x,y}$ denote the collection of hyperplanes separating $x$ from $y$. The \CAT distance between two points in a \CAT cube complex is multiplicatively related to $\mathcal{C}_{x,y}$:

\begin{lemma} [{\cite[Lemma 2.2]{caprice}}] \label{quasi-isometry between combinatorial and CAT(0) metrics}
Let $X$ be a CAT(0) cube complex, then there exists a constant $\cc$ depending only on the dimension of $X$,  such that for any $x,y \in X$, we have 

\begin{center}
    $d(x,y) \leq |\mathcal{C}_{x,y}| \leq \cc d(x,y)$.
\end{center}
\end{lemma}

We also need the following remark.

\begin{remark}\label{remark: convex hull}

If $B^{(1)}$ is a ball of radius $r$ in the $d^{(1)}$--metric, then the $d^{(1)}$ convex hull of  $B^{(1)}$ in $X^{(1)}$ has diameter at most $2vr.$
\end{remark}

\subsection{Counterexample to sublinear bound on separation}
Two hyperplanes are said to be \emph{$k$-separated} if the number of hyperplanes crossing them both is bounded above by $k$. In \cite{ChSu2014}, Charney and Sultan give the following characterization for uniformly contracting geodesic rays in CAT(0) cube complexes. Recall that a geodesic ray $b$ is said to be \emph{uniformly contracting} if it is $\kappa$-contracting with $\kappa=1$.
\begin{theorem}[{\cite[Theorem 4.2]{ChSu2014}}]\label{thm: Charney and Sultan hyperplanes crossing}
Let $X$ be a uniformly locally finite CAT(0) cube complex. There exist
$r > 0,$ $k \geq 0$ such that a geodesic ray $b$ in $X$ is uniformly contracting
if and only if $b$ crosses an infinite sequence of hyperplanes $h_1, h_2, h_3$, . . . at points   $x_i =
b \cap h_i$ satisfying:

\begin{enumerate}
    \item  $h_i
, h_{i+1}$ are $k$-separated and
\item $d(x_i,x_{i+1}) \leq r.$
\end{enumerate}

\end{theorem}

 It may be natural to expect that $\kappa$-contracting geodesic rays are characterized where both $k$ and $r$ are replaced by the $\kappa$-function. This is in fact not the case: 
\begin{example}\label{example: counter example}

Let $\Gamma_n$ be the graph of an $n$-depth binary tree, see Figure~\ref{fig:bt}. Note that $\Gamma_{n}$ embeds into $\Gamma_{n+1}$ by mapping the base-point of one to the other and extending upwards in the natural way.

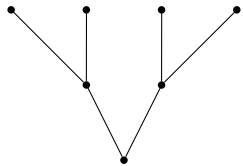
\begin{figure}[h]
\begin{tikzpicture}

\foreach[evaluate={\j=int(2^\y-1)}] \y in {0,1,2} 
	\foreach \x in {0,...,\j} 
		\node at (-\j/2-1/2+\x,\y)[scale=0.3,circle,fill](\x\y){};

\foreach[evaluate={\j=int(2^\y-1)}] \y in {0,1} 
	\foreach \x in {0,...,\j}{
		\pgfmathsetmacro{\PHI}{\y+1}
		\pgfmathsetmacro{\THETA}{int(2*\x)}
		\draw (\x\y) -- ($(\THETA\PHI) - (0.05,0)$);
		\pgfmathsetmacro{\THETA}{int(2*\x+1)}
		\draw (\x\y) -- ($(\THETA\PHI) - (0.05,0)$);
		}
\end{tikzpicture}
\caption{Binary tree $\Gamma_2$}
\label{fig:bt}
\end{figure}

We can now construct the cube complex, $X$. For each $n \in \mathbb{N} \cup \{0\},$ create the following sub-complexes $C_n = \Gamma_n \times [0,2^n]$. Next, for each $n$ glue the second face of $C_n$ (the face $\Gamma_n \times \{2^n\}$) to the first face of $C_{n+1}$ (the face $\Gamma_{n+1} \times \{0\}$) via the embedding $\Gamma_n \hookrightarrow \Gamma_{n+1}$.

Consider the geodesic ray $b(t)$ which is the bottom edge of the complex $X$. i.e. if we orient $X$ so that in each $C_n$ we have all the branches of the binary tree pointing upwards, the image of $b(t)$ is the bottom most edge of the complex (in particular, $b(0)$ is the bottom vertex of $\Gamma_0 \times \{0\})$.

We claim the following about the complex $X$ and $b(t)$:
\begin{enumerate}
\item Every point $x \in C_n$ is at most distance $n$ away from the geodesic $b$.
\item For any $t$, if $b(t) \in C_n$ then $2^{n}-1\leq t \leq 2^{n+1}-1.$
\item The geodesic ray $b$ is $\kappa$-contracting: (1) and (2) implies that the geodesic $b$ satisfies the condition of Lemma~\ref{lem: contracting implies projection contracting}. \noindent (a quick verification of this is to look at $X$ sidelong and the complex will be contained under the curve of $\log_2(x+1)$ which is sublinear). We present the proof of (3) in more details:

\begin{proof} Let $x,y \in X$ so that $d(x, y) \leq d(x,b)$. Now since $x \in X$ there is an $n$ with $x \in C_n$. Claim 1 says that $d(x, b) \leq n$, so we know that $d(x_b, y_b)\leq n$, but since $x_b \in C_n$ as well, by Fact 2 we see that $2^n -1 \leq d(b(0),x_b) = ||x_b||$. This gives us that $n \leq \log_2(||x_b|| + 1)$ and so $d(x_b, y_b) \leq \log_2(||x_b||+1)$. A little algebra implies the existence of a constant $\cc$ so that $d(x_b,y_b) \leq \cc \log_2(x_b)$  which shows that $b$ is $\log_2$-contracting, hence, we are done.

\end{proof}

\item If $h_{i-1}$ and $h_i$ are a pair of hyperplanes in $X$ which intersect $b(t)$ such that $h_{i-1}$ is in $C_n$ and $h_i$ is in $C_m$ with $n \leq m$, then $h_{i-1}$ and $h_i$ are at $(2^{n+1}-2)$--separated but are only $2n$-well separated.

\begin{proof} Since $h_{i-1}$ intersects $b$ it will be isometric to the graph $\Gamma_n$. The midpoint of every edge in $\Gamma_n$ corresponds to a hyperplane that intersects $h_{i-1}$, and this is all of them. Each of these hyperplanes also intersects $h_{i}$ since $h_i$ is either also in $C_n$ or in $C_m$ with $m > n$ and intersects $h_i$ at the same midpoints of the edges of $\Gamma_n \subset \Gamma_m$. Since there are exactly $2^{n+1}-2$ edges in $\Gamma_n$ then $h_{i-1}$ and $h_i$ are $(2^{n+1}-2)$--separated.

Now consider a collection of hyperplanes that intersect $h_{i-1}$. Since $h_{i-1}$ is isometric to $\Gamma_n$ we may consider the intersection points as points in $\Gamma_n$. If the convex hull of these points contains a tripod then there are three hyperplanes in the collection that form a facing triple. If the convex hull doesn't contain a tripod, then the points fall on a single geodesic line. The largest geodesic line inside of $\Gamma_n$ is length $2n$. Thus any collection of hyperplanes with no facing triples can only intersect $2n$ edges and so has cardinality at most $2n$.

\end{proof}

\end{enumerate}
\end{example}

\subsubsection{Proof of Theorem~\ref{theorem:counter example in the intro}} Let $\kappa(t)$ be any sublinear function, let $\cc$ be some constant, and let $\calH$ be a sequence of hyperplanes with 
\[b(t_i)\cap h_i\neq \emptyset  \text{ and } |t_{i}-t_{i+1}| < \cc\kappa(t_{i+1}).\] For each pair of hyperplanes $h_i$ and $h_{i+1}$, let $K(t_{i+1})$ be the smallest number such that they are $K(t_{i+1})$--separated.

For each $t_{i}$ there is a unique integer $n_{i}$, so that $n_{i} \leq \log_2(t_{i}+1) \leq n_{i}+1$. However, this implies the following inequality

\begin{equation} 2^{n_{i}} - 1 \leq t_{i} \leq 2^{n_{i}+1} - 1.
\label{inequality bounds}
\end{equation}

\vspace{12pt}

But by construction, this is simply the statement that $b(t_{i}) \in C_{n_{i}}$ and thus that $h_{i} \in C_{n_{i}}$. It follows immediately that $h_{i+1}$ is in $C_m$ for $m \geq n_i$. We can then apply fact (4) to get that $h_i$ and $h_{i+1}$ are $(2^{n_i+1}-2)$--separated. This tells us that $K(t_{i+1}) = 2^{n_i+1}-2$.

 Using inequality (\ref{inequality bounds}) we get that $t_i-1 \leq K(t_{i+1})$, or in other words that $t_i \leq K(t_{i+1}) + 1$. In particular this means that $ 1 \leq \lim_{t_i \to \infty} \frac{K(t_{i+1})}{t_i}$, and so $K(t_{i+1})$ is asymptotically linear. Since $t_{i+1}-t_i \leq \cc \kappa(t_{i+1})$, and since $\kappa$ is a sublinear function, we can choose $i$ large enough so that $t_{i+1} \leq t_i + \cc \kappa(t_{i+1})\leq t_i+ \cc \frac{t_{i+1}}{2\cc}$ and hence $\frac{t_{i+1}}{2} \leq t_i$ or $t_{i+1} \leq 2t_{i}$. This implies that

 $$ 1 \leq \lim_{t_i \to \infty} \frac{K(t_{i+1})}{t_i} \leq \lim_{t_i \to \infty} \frac{2K(t_{i+1})}{t_{i+1}},$$ and hence $\lim_{t_i \to \infty} \frac{K(t_{i+1})}{t_{i+1}} \geq \frac{1}{2}$. Since $\kappa(t_{i+1})$ is a sublinear function, we get that $K(t_{i+1})$ strictly dominates $\kappa(t_{i+1})$ for large enough values of $t_i$.

On the other hand, by Lemma~\ref{lem: contracting implies projection contracting}, $b(t)$ is a $\log_2$-contracting geodesic. More precisely, if we take the sequence of hyperplanes that intersect $b(t)$ we can see that $h_i$ and $h_{i+1}$ are at most $2\log_2(t_i+1)$-well separated, since the longest chain of hyperplanes which intersect both is at most twice the height of the binary tree that $h_i$ lives inside.

\subsection{The characterization theorem.}
Now we are ready to present and prove Theorem~\ref{thm: 3rd theorem in the intro} in the introduction. 

\begin{definition}[Facing Triples]\label{def:facing triples}
A collection of three hyperplanes $h_1, h_2, h_3$ is said to form a \emph{facing triple} if they are disjoint and none of the three hyperplanes separates the other two.
\end{definition}

   Notice that if a (combinatorial or CAT(0)) geodesic $b$ crosses three disjoint hyperplanes $h_1,h_2$ and $h_3$ respectively, then $h_2$ separates $h_1$ and $h_3.$ In particular, a geodesic $b$ cannot cross a facing triple. Conversely, if $\mathcal{C}$ is a collection of hyperplanes  which contains no facing triple, then there is a geodesic which crosses a \emph{definite proportion} of the hyperplanes in $\mathcal{C}$:
    
    \begin{lemma}[{\cite[Corollary 3.4]{Hagen20}}]\label{lem:definite proportion} Let $X$ be a CAT(0) cube complex of dimension $v$. There exists a constant $k$, depending only on $v$, such that the following holds.  If $\mathcal{C}$ is a collection of hyperplanes which contains no facing triple, then there exists a (combinatorial or CAT(0)) geodesic which crosses at least $\frac{\mathcal{|C|}}{k}$ hyperplanes from the collection $\mathcal{C}.$
    
    \end{lemma}
 
\begin{definition}[Well-separated hyperplanes]
Two disjoint hyperplanes $h_1, h_2$ are said to be \emph{$k$-well-separated} if any collection of hyperplanes intersecting both $h_1, h_2$, and which does not contain any facing triple, has cardinality at most $k$. We say that $h_1$ and $h_2$ are \emph{well-separated} if they are $k$-well-separated for some $k.$
\end{definition}

\begin{definition}[Combinatorial projections]

Let $X$ be a CAT(0) cube complex and let $X^{(1)}$ be its 1-skeleton with the $d^{(1)}$-metric. Let $Z$ be a convex sub-complex. For all $x \in X^{(0)}$, the point $\sg_Z(x) \in Z$ is defined to be the closest vertex in $Z$ to $x$. The vertex $\sg_Z(x)$ is unique and it is characterized by the property that a hyperplane $h$ separates $\sg(x)$ from $x$ if and only if $h$ separates $x$ from $Z$, for more details see \cite{Behrstock2017}.
\end{definition}

The following proposition states that for a convex sub-complex $Z$ and for $x,y \notin Z$, if a hyperplane $h$ separates $x$ and $y$ then either $h$ doesn't intersect $Z$, or if does, it separates $\sg_Z(x)$ and $\sg_Z(y).$

 \begin{proposition}[{\cite[Proposition 2.7]{Genevois2020}}]\label{key: hyperplanes intersecting a convex subcomplex}
 Let $Z$ be a convex sub-complex in a CAT(0) cube complex $X$ and let $x,y $ be two vertices of $X^{(1)}$. The hyperplanes separating $\sg_Z(x)$ and $\sg_Z(y)$ are precisely the ones separating $x$ and $y$ which intersect $Z.$

\end{proposition}

We proceed with a few lemmas working towards the forward direction of the argument.

%
%
%

\begin{lemma}\label{lem: insert a hyperplane in the middle}
Let $X$ be a CAT(0) cube complex of dimension $v$ and suppose $b$ is a $\kappa$-contracting geodesic ray in $X$. There exists a constant $\cc$ such that if $t, s \in \mathbb{R}$ are such that $t-s \geq \cc \kappa(t)$, then $[b(s), b(t)]$ crosses a hyperplane whose nearest-point projection to $b$ is entirely between $b(s)$ and $b(t).$

\end{lemma}

\begin{proof}
Let $\mathcal{C}$ be the collection of hyperplanes which cross $[b(s), b(t)]$. Let $h \in \mathcal{C},$ and define $z=h \cap b$. If $h$ projects outside $[b(s), b(t)]$, then it contains a point $x$ with $x_b=b(t)$ or $x_b=b(s)$. If $h$ contains a point $x$ with   $x_b=b(t)$, then using $\kappa$-slimness of the triangle $\Delta=\Delta (x, b(t),z),$ there exists a constant $\cc_1$ such that $h \cap B_{\cc_1 \kappa(t)}(b(t)) \neq \emptyset.$ Since the CAT(0) metric is
bounded above and below by linear functions of the $d^{(1)}$-metric, depending only on the dimension $v$, there exists a constant $\cc_2=\cc_2(v, \cc_1)$ such that the ball $B_{\cc_1 \kappa(t)}(b(t))$ is contained in a combinatorial ball  $B^{(1)}$ of radius $\cc_2 \kappa(t)$.

By Remark \ref{remark: convex hull}, the diameter of the convex hull $Z$ of $B^{(1)}$ is at most $2v\cc_2 \kappa(t),$ where $v$ bounds the dimension of the cube complex.

 Let $x, y \in X^{(1)}$ be two vertices which are separated by every hyperplane in $\mathcal{C}$. Since $Z$ has diameter at most $2v\cc_2 \kappa(t)$, we have \[d(\sg_{Z}(x), \sg_{Z}(y)) \leq 2v\cc_2 \kappa(t).\] Also, using Proposition~\ref{key: hyperplanes intersecting a convex subcomplex}, any hyperplane $h \in \mathcal{C}$ that crosses $Z$ separates $\sg_{Z}(x), \sg_{Z}(y)$. Thus, the number of such hyperplanes is at most $diam(Z) \leq 2v\cc_2 \kappa(t)$.

  Therefore, $2v\cc_2 \kappa(t)$ bounds the number of hyperplanes $h \in \mathcal{C}$ that have a projection point to $b(t).$ Similarly, $2v\cc_2 \kappa(s)$ bounds the number of hyperplanes $h \in \mathcal{C}$ which project to $b(s)$. In turns, the number of hyperplanes $h \in \mathcal{C},$ projecting outside $[b(s), b(t)]$ is at most $2v\cc_2 (\kappa(s)+\kappa(t)) \leq 4v\cc_2 \kappa(t),$ where $\cc_2$ depends only on $v$ and $\cc_1.$ Therefore, choosing $\cc$ so that the number of hyperplanes crossed by $[b(s),b(t)]$ exceeds $4v\cc_2$ gives the desired result.

\end{proof}

\begin{lemma} \label{lemma: travelling close to geodesics}
Let $b$ be a $\kappa$-contracting geodesic ray with constant $\cc$ and let $x,y \in X$ and not in $b$ such that $d(y_b,b(0)) \geq d(x_b,b(0))$. If $d(x_b,y_b) \geq 3\cc \kappa(y_b),$ then we have the following:

$$d(x,y) \geq d(x, x_b)+d(x_b, y_b)+d(y_b, y)-\cc \kappa(x_{b})-\cc \kappa(y_{b}).$$

\end{lemma}

\begin{proof}
First, we will show that $d(x,y) \geq d(x,x_b)+d(y,y_b).$ For the sake of contradiction, suppose that $d(x,y) < d(x,x_b)+d(y,y_b).$ This implies the existence of a point $z \in [x,y]$ such that $d(z,x)<d(x,x_b)$ and $d(z,y)<d(y,y_b).$ Since our geodesic $b$ is $\kappa$-contracting, we obtain that $d(x_b, z_b) \leq \cc \kappa(x_b)$ and $d(y_b,z_b) \leq \cc \kappa(y_b)$. This implies that $$d(x_b,y_b) \leq d(x_b, z_b)+ d(z_b, y_b) \leq \cc \kappa(x_b)+ \cc \kappa(y_b) \leq 2\cc \kappa(y_b),$$ contradicting the assumption that $d(x_b,y_b) \geq 3\cc \kappa(y_b).$ Thus we have shown that 
\[
d(x,y) \geq d(x,x_b)+d(y,y_b).
\]

Let $x'$ be the point on $[x,y]$ at distance $d(x, x_b)$ from $x$. Similarly, let $y'$ be the point on $[x,y]$ at distance $d(y, y_b)$ from $y$. The projections of $[x,x']$ and $[y',y]$ to $b$ have diameters at most  $\cc \kappa(x_b)$ and $\cc \kappa(y_b)$, respectively. Since projections in CAT(0) are distance decreasing, we have $$d(x',y') \geq d(x'_b, y'_b) \geq d(x_b, y_b)-\cc\kappa(x_b)-\cc \kappa(y_b).$$ 

Now, we have 
\begin{align*}
d(x,y)&=d(x,x')+d(x',y')+d(y',y) \\
         &=d(x, x_b)+d(x',y')+d(y, y_b) \\
         & \geq d(x, x_b)+d(x_b, y_b)-\cc\kappa(x_b)-\cc \kappa(y_b)+d(y, y_b).
         \end{align*}

\end{proof}

\begin{lemma} \label{lemma: large projection means close}
Let $b$ be a $\kappa$-contracting geodesic ray  starting at $\go$ with constant $\cc$. Let $x,y \in X$ and not in $b$ such that $d(\go,x_b) \leq d(\go,y_b)$ and $d(x_b,y_b) \geq 4 \cc \kappa(y_b).$ We have the following:

\begin{enumerate}
    \item  $d(x,y) \geq d(x, x_b)+d(x_b, y_b)+d(y_b, y)-2\cc \kappa(y_{b}).$
\item  $\pi_{b}([x,y]) \subseteq N_{5\cc\kappa(y_b)}([x,y]). $

\end{enumerate}

\end{lemma}

\begin{figure}[H]
\begin{tikzpicture}[scale=0.4]
\tikzstyle{vertex} =[circle,draw,fill=black,thick, inner sep=0pt,minimum size=.5 mm]
 
[thick, 
    scale=1,
    vertex/.style={circle,draw,fill=black,thick,
                   inner sep=0pt,minimum size= .5 mm},
                  
      trans/.style={thick,->, shorten >=6pt,shorten <=6pt,>=stealth},
   ]

%
%
%

%
%

\draw[->, thick, black] (-10, 0) to (10, 0){};
\node [vertex] at (-8, 8)[label=left:$x$]{};
\node [vertex] at (8, 8)[label=right:$y$]{};

\draw[-, thick, black] (-8, 8) to [bend left=15](-6, 0){};
\draw[-, thick, black] (8, 8) to [bend right=15](6, 0){};
\node [vertex] at (-6, 0)[label=below:$x_{b}$]{};
\node [vertex] at (6, 0)[label=below:$y_{b}$]{};
{};
 \draw [decorate,decoration={brace,amplitude=10pt},xshift=0pt,yshift=0pt]
  (6,0) -- (-6,0)  node [thick, black,midway,xshift=0pt,yshift=0pt] {};      
 \node at (0.5, -1.4) {$d(x_{b}, y_{b}) \geq 4 c \kappa (y_{b})$};
  \node at (0, 2) {$[x, y]$};

\draw[dashed, thick, black] (-6, 0) to(-4, 2.1){};
\draw[dashed, thick, black] (6, 0) to(4, 2.1){};
\draw[thick] (-8,8) to (-4.5, 3){};
\draw[thick] (8,8) to (4.5, 3){};
\draw[thick] (-4.5, 3) to [bend right=59] (4.5, 3){};

%
\end{tikzpicture}
\caption{$d(x_{b}, [x, y]) \leq 5 c \kappa (x_{b}), d(y_{b}, [x, y]) \leq 5 c \kappa (x_{b}).$}
\end{figure}
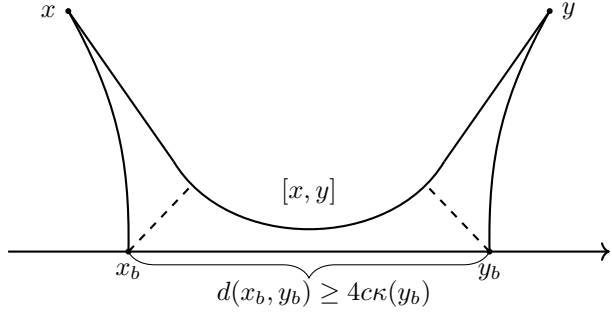

\begin{proof}

  The first part follows from Lemma \ref{lemma: travelling close to geodesics} since $\kappa$ is a non-increasing function. Let $x', y'$ be as in the previous problem. This gives that

     \begin{align*}
        d(x,y) &= d(x,x')+d(x',y')+d(y',y) \\
               &= d(x, x_b)+d(x',y')+d(y, y_b) \\
               & \leq d(x, x_b)+d(x_b, y_b)+d(y, y_b).
      \end{align*}

     Thus, $d(x',y') \leq d(x_b, y_b).$ Now, applying part (1)  to $x',y'$, we get that
     
     \begin{align*}
     d(x_b, y_b) &\geq d(x',y') \\
                 &\geq d(x', x'_b)+d(x'_b, y'_b)+d(y', y'_b)-2\cc\kappa(y'_b) \\
                 &\geq d(x', x'_b)+d(x_b, y_b)-2\cc\kappa(y_b)+d(y', y'_b)-2\cc\kappa(y'_b)\\
                 &\geq d(x', x'_b)+d(x_b, y_b)-2\cc\kappa(y_b)+d(y', y'_b)-2\cc\kappa(y_b) \\
                 &\geq d(x', x'_b)+d(x_b, y_b)+d(y', y'_b)-4\cc\kappa(y_b).
       \end{align*}

     This yields that $d(x', x'_b)+d(y', y'_b) \leq 4\cc\kappa(y_b)$. In particular,
     
     \begin{center}
         
    max$\{d(x', x'_b), d(y',y'_b)\} \leq 4\cc \kappa(y_b).$
     \end{center}
    Using convexity of the CAT(0) metric, we get that $\pi_{b}([x',y']) \subseteq N_{4\cc\kappa(y_b)}([x,y])$, but since every point in $[x_b, y_b]$ is within $\cc \kappa(y_b)$ of $[x'_b, y'_b]$, the result follows. 
   
    \end{proof}

The proof of Theorem \ref{thm: main theorem section 4} is presented in two parts.

\subsubsection*{The forward direction of Theorem \ref{thm: main theorem section 4}: $\kappa$-contracting implies $\kappa$-excursion}

Let $b$ be a $\kappa$-contracting geodesic ray with constant $\cc$ and let $\dd$ be a constant as in Lemma \ref{lem: insert a hyperplane in the middle}. Since $\kappa$ is a sublinear function, there exists an $s_0 \geq 0$ such that if $s \geq s_0,$ we have $\kappa(s) \leq \frac{s}{2\dd}$. Choose $s_1 \geq s_0$ and divide the geodesic ray into pieces $0<s_1<s'_1<s_2<s'_2<...<s_i<s'_i<..$ such that for each $i$, we have

\begin{align*}
    s'_{i}-s_i  &= \dd \kappa(s'_{i}), \\
    s_{i+1}-s'_i&= 4\cc(1+2\dd) \kappa(s_{i+1}).
\end{align*}

Therefore,  $s'_{i} =s_i+\dd \kappa(s'_{i}) \leq s_i+\dd\frac{s'_{i}}{2\dd}$ and hence $\frac{s'_{i}}{2} \leq s_i$ or $s'_{i} \leq 2 s_i.$

Lemma \ref{lem: insert a hyperplane in the middle} assures that each $[b(s_i),b(s_i')]$ crosses a hyperplane $h_i$ whose projection lies entirely inside $[b(s_i),b(s_i')]$. Define $b(t_{i}):=h_{i} \cap [b(s_i),b(s_i')]$. Since

\begin{align*}
    d(b(t_i), b(t_{i+1})) &\leq d(s_i, s'_i)+d(s'_i,s_{i+1})+d(s_{i+1}, s'_{i+1}) \\
    &\leq \dd \kappa(s'_i)+4\cc(1+2\dd)\kappa(s_{i+1})+\dd \kappa(s'_{i+1})\\
    &\leq \dd \kappa(s'_i)+4\cc(1+2\dd)\kappa(s_{i+1})+2\dd \kappa(s_{i+1}) \\
    &\leq \dd \kappa(t_{i+1})+4\cc(1+2\dd)\kappa(t_{i+1})+2\dd \kappa(t_{i+1})\\
    &= (3\dd +4\cc(1+2\dd))\kappa(t_{i+1}).
\end{align*}

Let $x,y$ be in $\in h_i, h_{i+1}$ respectively. Notice that as  $y_b \in [b(s_{i+1}),b(s'_{i+1}) ]$, we have $$||y_b|| \leq s'_{i+1} \leq 2s_{i+1}<(1+2\dd) s_{i+1}.$$ Hence, we have

\begin{center}
    
$\kappa(||y_b||) \leq \kappa((1+2\dd) s_{i+1} ) \leq (1+2\dd)\kappa(s_{i+1}) \leq (1+2\dd)\kappa(t_{i+1}).$

\end{center}

Notice that diam$ (\pi_b ([x,y]) \geq  4\cc(1+2\dd) \kappa(s_{i+1}) \geq 4\cc  \kappa(||y_b||)=4\cc \kappa(||y_b||).$ Using Lemma \ref{lemma: large projection means close}, any geodesic $[x, y]$ connecting $h_i$ to $h_{i+1}$ has to come $5\cc \kappa(y_b)$ close to every point in $[t_i,t_{i+1}]$. In particular, we have $[x,y] \cap B_{5\cc \kappa (||y_b||)}(t_{i+1}) \neq \emptyset$. Hence, using the equation above, we have

\begin{center}

$[x,y] \cap B_{5\cc(1+2\dd) \kappa(t_{i+1})}(t_{i+1}) \neq \emptyset.$
\end{center}

Let $\cc'=5\cc(1+2\dd)$. The previous equation then implies that every hyperplane $h$ intersecting both $h_i$ and $h_{i+1}$ meets $B_{\cc' \kappa(t_{i+1})}(t_{i+1})$. Thus, if $\mathcal{C}$ denotes the collection of hyperplanes meeting both $h_i, h_{i+1}$ which contain no facing triple, then $|\mathcal{C}|< \infty.$ Since the CAT(0) metric is quasi-isometric to the $d^{(1)}$-metric, there exists a constant $\cc_1=\cc_1(v, \cc')$ such that the ball $B_{\cc' \kappa(t_{i+1})}(t_{i+1})$ is contained in a combinatorial ball $B^{(1)}$ of radius $\cc_1 \kappa(t_{i+1})$.

Using Remark \ref{remark: convex hull}, the diameter of the convex hull $Z$ of $B^{(1)}$ is at most $2v\cc_1 \kappa(t_{i+1}),$ where $v$ bounds the dimension of the cube complex.

 By Lemma \ref{lem:definite proportion}, there exist two points  $x',y' \in X^{(1)}$ and $\frac{1}{k}|\mathcal{C}|$ hyperplanes in $\mathcal{C}$ which separate $x',y'$, where $k$ is a uniform constant depending only on the dimension $v$. Denote this collection by $\mathcal{C}'$. Since $Z$ has diameter at most $2v\cc_1 \kappa(t_{i+1})$, we have $d(\sg_{Z}(x'), \sg_{Z}(y')) \leq 2v\cc_1 \kappa(t_{i+1})$. Also, using Proposition \ref{key: hyperplanes intersecting a convex subcomplex}, any hyperplane $h \in \mathcal{C}'$ crossing $Z$ separates $\sg_{Z}(x'), \sg_{Z}(y')$. Thus, there can be at most $diam(Z) \leq 2v\cc_1 \kappa(t_{i+1})$ such hyperplanes. Therefore, $2v\cc_1 \kappa(t_{i+1})$ bounds the number of hyperplanes $h \in \mathcal{C}'$ which in turn implies that  $2kv\cc_1 \kappa(t_{i+1})$ bonds in the number of hyperplanes in $\mathcal{C}.$

\subsubsection{The backward direction of Theorem~\ref{thm: main theorem section 4}: $\kappa$-excursion implies $\kappa$-contracting.}

Let $b$ be the $\kappa$--excursion geodesic under consideration. Let $[x,y]$ be a geodesic connecting $x \in X$ to  $y \in b $. The following lemma states that if $[x,y]$ meets one of the separated hyperplanes crossed by $b$, it meets it at distance at most $\kappa$ from $b$. The following argument is inspired by the similar argument in \cite{Merlin}.

\begin{lemma} \label{lem: Main lemma to show backward direction}
Let $b$ be a $\kappa$-excursion geodesic and let $\{h_i\}$ be the sequence of well-separated hyperplanes crossed by $b$. There exists a constant $\dd$, depending only on $v$ and the excursion constant $\cc$, such that the following holds:

\begin{enumerate}
    \item If $x \in h_i$ and $y \in b$, with $y>b(t_{i+2})$, then the geodesic $[x,y]$ intersects $h_{i+1}$ at distance at most $\dd \kappa(t_{i+1})$ from $b(t_{i+1})$.
    
    \item If $x \in h_i$ and $y \in b$, with $y<b(t_{i-2})$, then the geodesic $[x,y]$ intersects $h_{i-1}$ at distance at most $\dd \kappa(t_{i-1})$ from $b(t_{i-1})$.

\end{enumerate}
\end{lemma}

\begin{figure}[H]
\begin{tikzpicture}[scale=0.3]
\tikzstyle{vertex} =[circle,draw,fill=black,thick, inner sep=0pt,minimum size=.5 mm]
 
[thick, 
    scale=1,
    vertex/.style={circle,draw,fill=black,thick,
                   inner sep=0pt,minimum size= .5 mm},
                  
      trans/.style={thick,->, shorten >=6pt,shorten <=6pt,>=stealth},
   ]

%
%
%

%
%

\begin{scope}
        \myGlobalTransformation{-100}{75};
        \draw [black!50, step=2cm] grid (6, 12);
        \draw [thick, draw=black, fill=yellow, fill opacity=0.5]
       (0,0) -- (0,12) -- (6,12) --(6,0)--cycle;
\end{scope}

\begin{scope}
        \myGlobalTransformation{0}{75};
         \draw [black!50, step=2cm] grid (6, 12);
        \draw [thick, draw=black, fill=yellow, fill opacity=0.5]
       (0,0) -- (0,12) -- (6,12) --(6,0)--cycle;
\end{scope}

\begin{scope}
        \myGlobalTransformation{100}{75};
        \draw [black!50, step=2cm] grid (6, 12);
        \draw [thick, draw=black, fill=yellow, fill opacity=0.5]
       (0,0) -- (0,12) -- (6,12) --(6,0)--cycle;
\end{scope}

\draw[->, black] (-7, 10) to (15, 10){};
\draw [-] (-5, 18) to [bend right=15] (10, 10);
\node [vertex] at (10, 10)[label=above right:$y$]{};
\node [vertex] at (-2.7, 16)[label=above right:$x$]{};
\node [vertex] at (0.9, 13.45)[label=above right:$x_{1}$]{};
\node [vertex] at (4, 11.9)[label=above right:$x_{2}$]{};

\draw [-, thick] (-5, 12) to  (-1, 12);
\node [vertex] at (-5, 12)[label=left:$C'_{1}$]{};
\draw [-, thick] (-1, 12) to  [bend right=65] (-1.8, 11);

\draw [-, thick] (-1.8, 11) to (-1.5, 6);
\node [vertex] at (-1.5, 6)[label=below:$C'_{2}$]{};

\draw [-, thick] (-1, 12) to [bend right=45] (3, 16);
\node at (3, 16)[label=above:$C_{2}$]{};
\draw [-, thick] (-1, 11.2) to  (5, 10.3);
\node at (5, 10.3)[label=below right:$C_{2}$]{};
\draw [-, thick] (-1, 10.5) to [bend left=45] (3, 6);
\node at (3, 6)[label=below:$C_{1}$]{};

%
%
%
%
%
%
%
%
\node at (-2, 1)[label=left:$h_{i}$]{};
\node at (3, 1)[label=left:$h_{i+1}$]{};
\node at (6.5, 1)[label=left:$h_{i+2}$]{};
\end{tikzpicture}
\end{figure}

%
%
%
%
%

\begin{proof}

We give a detailed proof to (1), the proof for (2) is identical.  Let $\cc$ be the excursion constant of $b$ and let $\cc_1$ be a constant as in Lemma \ref{quasi-isometry between combinatorial and CAT(0) metrics}. Define $x_1=[x,y] \cap h_{i+1}$, $x_2=[x,y] \cap h_{i+2}$ and let $\mathcal{C}$ denote the collection of hyperplanes intersecting $[x_1,b(t_{i+1})]$. The collection $\mathcal{C}$ contains no facing triple since every hyperplane in $\mathcal{C}$ is crossed by a geodesic. The collection $\mathcal{C}$ of hyperplanes intersects $[b(t_{i+1}), b(t_{i+2})]$,  $[x_2,b(t_{i+2})]$ or $[x_1,x_2]$. Denote such hyperplanes by $\mathcal{C}_1, \mathcal{C}_2$ and $\mathcal{C}_3$ respectively.

By Lemma \ref{quasi-isometry between combinatorial and CAT(0) metrics}, the number of hyperplanes in $\mathcal{C}$ intersecting both $[x_1,b(t_{i+1})]$ and $[b(t_{i+1}), b(t_{i+2})]$ is bounded above by $\cc_1 \cc \kappa(t_{i+2}).$

Since $h_{i+1}, h_{i+2}$ are $\cc\kappa(t_{i+2})$-well-separated, there is at most $\cc \kappa(t_{i+2})$ hyperplanes in $\mathcal{C}$ that intersect $[x_2,b(t_{i+2})]$. Therefore, we have:

\begin{align*}
    |\mathcal{C}| & \leq |\mathcal{C}_1| +|\mathcal{C}_2| +|\mathcal{C}_3|\\
                  & \leq \cc_1 \cc \kappa(t_{i+2})+ \cc \kappa(t_{i+2})+|\mathcal{C}_3| \\
                  &   = (\cc_1\cc+\cc)\kappa(t_{i+2})+|\mathcal{C}_3|.\\
\end{align*}

Every hyperplane in $\mathcal{C}_3$ intersects $[b(t_i),b(t_{i+1})]$ or $[x,b(t_i)]$. Denote such hyperplanes by $\mathcal{C}'_1, \mathcal{C}'_2$, respectively.

 Using Lemma \ref{quasi-isometry between combinatorial and CAT(0) metrics}, $|\mathcal{C}'_1| \leq \cc_1\cc \kappa(t_{i+1})$. Since $h_i, h_{i+1}$ are $\kappa (t_{i+1})$-well-separated, $|\mathcal{C}'_2| \leq \cc \kappa(t_{i+1}).$ Therefore,
 \begin{center}
     
 $|\mathcal{C}_3| \leq \cc_1\cc \kappa(t_{i+1})+\cc \kappa(t_{i+1})$
 
  \end{center}
  
  This implies the following:
  
  \begin{align*}
      \mathcal{|C|} &\leq (\cc_1\cc+\cc)\kappa(t_{i+2})+|\mathcal{C}_3| \\
                  & \leq (\cc_1\cc+\cc)\kappa(t_{i+2})+ \cc_1\cc \kappa(t_{i+1})+\cc \kappa(t_{i+1})\\
                  & \leq (\cc_1\cc+\cc)\kappa(t_{i+2})+ \cc_1\cc \kappa(t_{i+2})+\cc \kappa(t_{i+2})\\
                  & \leq 2(\cc_1\cc+\cc ) \kappa(t_{i+2}).
  \end{align*}

Since $\mathcal{C}$ was defined to be the collection of hyperplanes separating intersecting the geodesic $[x_1, b(t_{i+1})]$, using Lemma \ref{quasi-isometry between combinatorial and CAT(0) metrics}, we have $d(x_1,b (t_{i+1})) \leq 2\cc_1(\cc_1\cc+\cc ) \kappa(t_{i+2}).$

\end{proof}

From Lemma~\ref{lem: Main lemma to show backward direction} we have the following:
\begin{corollary}\label{cor-distancebound}
Let $b$ be a $\kappa$-excursion geodesic and let $\{h_i\}$ be the sequence of well-separated hyperplanes crossed by $b$. There exists a constant $\dd$, depending only on $v$ and the excursion constant $\cc$, such that the following holds:

\begin{enumerate}
    \item For any $x$ between $h_{i-1}$ and $h_{i}$, if $y \in b$ satisfies $y>b(t_{i+2})$, then the geodesic $[x,y]$ intersects $h_{i+1}$ at distance at most $\dd \kappa(t_{i+1})$ from $b(t_{i+1})$.
    
    \item For any $x$ between $h_i$ and $h_{i+1}$, if $y \in b$ satisfies $y<b(t_{i-2})$, then the geodesic $[x,y]$ intersects $h_{i-1}$ at distance at most $\dd \kappa(t_{i-1})$ from $b(t_{i-1})$.

    \item If $x$ lies between $h_i$ and $h_{i+1}$, then $x_b \in B_{\dd \kappa(t_{i})}(b(t_{i}))$.

\end{enumerate}
\end{corollary}

\begin{proof}
For (1), define $x':= [x,y] \cap h_i$. Applying Lemma \ref{lem: Main lemma to show backward direction} to the geodesic $[x',y]$ gives the desired claim. The proof of (2) is the exact same. For (3), if $x_b \in [b(t_{i-2}), b(t_{i+2})]$, then we are done since $b$ is a $\kappa$-excursion geodesic. Otherwise if $x_b \notin [b(t_{i-2}), b(t_{i+2})]$, then applying $(1)$ and $(2)$ to the geodesic $[x,x_b]$ yields the desired result.

\end{proof}

\subsubsection*{Proof of the backward direction of Theorem~\ref{thm: main theorem section 4}}
Let $x \in X$ and let $y \in b$, we will show that $b$ is $\kappa$-slim with respect to condition 1. In other words, we need to show there exists a constant $\cc$ such that $d(x_b, [x,y]) \leq \cc \kappa (x_b),$ which follows from Corollary~\ref{cor-distancebound}. Therefore, every $\kappa$-excursion geodesic is $\kappa$-slim and is hence $\kappa$-contracting.

\subsection{Consequence for right-angled Artin groups}

 Recall that two hyperplanes $h_1, h_2$ are \emph{ strongly separated} if they are disjoint and no hyperplane intersects both $h_1$ and $h_2$. By \cite{Genevois2020}, when the space in question is the Salvetti complex of a right-angled Artin group, two hyperplanes are strongly separated if and only if they are $k$-well-separated for any $k$. Based on this fact we have the following:

\begin{corollary} \label{hyperplane crossings in RAAG}

Let $X$ be the Salvetti complex complex of a right-angled Artin group, then a geodesic ray $b$ in $X$ is $\kappa$-contracting if and only if there exists constants $\cc$ and an infinite sequence of strongly separated disjointed hyperplanes that $b$ crosses, labelled $h_{1}, h_{2}, ...$ at points $b(t_i)$ , such that $d(b (t_i), b (t_{i+1}))=d(t_i, t_{i+1})< \cc \kappa(t_{i+1})$.

\end{corollary}

\begin{proof}
In a Salvetti complex, two hyperplanes are well-separated if and only if they are strongly separated, thus the corollary follows from Theorem \ref{thm: main theorem section 4}.
\end{proof}

\subsection{Progress in contact graphs associated with Salvetti complexes.}
In \cite{Hagen2013}, Hagen introduced the contact graph of a CAT(0) cube complex. It is defined as follows. Recall that a \emph{hyperplane carrier}  is the union of all (closed) cubes intersecting a hyperplane $h$, which we denote $N(h)$.

\begin{definition}[Contact graph] Let $X$ be a CAT(0) cube complex, the contact graph of $X$, denoted by $\mathcal X$, is the intersection
graph of the hyperplane carriers. In other words, it is a graph where vertices are hyperplane carriers in $X$, and two vertices are connected by an edge if the corresponding hyperplanes carriers intersect. In this case we can also say the two hyperplanes have \emph{contact relation}.

\end{definition}

For a CAT(0) cube complex $X$, there is a coarse map $p: X \rightarrow \mathcal{C}X$: Given $x \in X^{(0)}$, the set of hyperplanes $h$ with $x \in N (h)$ corresponds to a complete subgraph
of $\mathcal{C}X$, which we denote $p(x)$. Hence $p : X^{(0)} \rightarrow
\mathcal{C}X$ is a coarse map. For each edge $e$ of $X$, we
define $p(e)$ to be the vertex corresponding to the hyperplane dual to $e$. Hence we have a coarse
map $p : X^{(1)} \rightarrow
\mathcal{C}X$. More generally, if $c$ is an open $n$--cube, $n \geq 1$, let $p(c)$ be the complete
subgraph with vertex set the hyperplanes intersecting $c$.

As an application to Theorem~\ref{thm: main theorem section 4}, we will prove the Corollary~\ref{cor:final statement of the intro} from the introduction. But first we need a lemma:

\begin{lemma}\label{lem:projection length}
Let $b$ be a geodesic ray in a finite dimensional $\CAT$ cube complex and let $p:X \rightarrow \mathcal{C}X$ be the projection defined in the previous paragraph. If $b$ intersects a sequence of strongly separated hyperplanes, $h_1, h_2, ..., h_n$ then the diameter of the projection of $b$ in the contact graph $\mathcal{C}X$ is at least $n-1$.
\end{lemma}

\begin{proof}
Reindexing the $h_i$ if necessary, we may assume that the hyperplanes intersect the geodesic $b$ in the same order as their indexing numbers. 

For the sake of contradiction assume that the diameter of the projection of $b$ in $\mathcal{C}X$ is less than $n-1$. In particular, since $h_1$ and $h_n$ are in the projection, we know that the distance between $h_1$ and $h_n$ is less than $n-1$. Let $v_1=h_1, v_2, ... ,v_{k-1}, v_k=h_n$ be a sequence of vertices in $\mathcal{C}X$ which forms an edge path geodesic from $h_1$ to $h_n$. The assumption tells us the edge path is no longer than $n-1$ i.e. that $k-1<n-1$, or in other words that $k < n$. 

Since the sequence of $v_j$ forms a geodesic in $\mathcal{C}X$ we know that for each $j$ the hyperplane carrier of $v_j$ intersects the hyperplane carrier of $v_{j+1}$. Thus, in the space $X$, there is a continuous path from the intersection point $b \cap h_1$ to $b\cap h_n$ which is the concatination of line segments contained inside a hyperplane carrier of $v_j$ for some $j$. Lets call this continuous path $c$. 

Note that $h_1$ and $h_n$ intersect $c$ (at the starting point and ending point respectively). Also notice that the sequence of planes $h_1, ..., h_n$ cannot contain any facing triples since they all intersect the same geodesic (see the discussion after Definition~\ref{def:facing triples}). Thus for each $h_i$ with $1<i<n$, $h_i$ separates $h_1$ from $h_n$ and since $c$ is continuous, this means that $h_i$ must intersect $c$. 

Since for any $i$, $h_i$ intersects $c$ there must be a $j$ where $h_i$ meets the hyperplane carrier of $v_j$. Notice that a hyperplane meets another hyperplane's carrier if and only if both hyperplanes intersect some cube, but all hyperplanes in a given cube must intersect one another, and so for each $1\leq i \leq n$, $h_i$ intersects at least one of the $v_j$. Since $k < n$ there is a $1\leq j \leq k$ and $1\leq i \neq l\leq n$ so that $v_j$ intersects both $h_i$ and $h_l$.  

However since the $h_i$ are all strongly separated no $v_j$ intersects more than one of the $h_i$. This gives us the needed contradiction and so the diameter of the projection of $b$ is at least $n-1$.
\end{proof}

Now we can prove Corollary~\ref{cor:final statement of the intro} from the introduction in the case where $X$ is the Salvetti complex of a right-angled Artin group:

\begin{proposition}

Let $\Gamma$ be a graph, $A_\Gamma, \Tilde{\mathcal{S}_\Gamma}$ be the corresponding right-angled Artin group and  Salvetti complex respectively. Every $\kappa$-contracting geodesic makes a definite progress in the contact graph. More precisely, if $b$ is a $\kappa$-contracting geodesic ray and $p:\Tilde{\mathcal{S}_\Gamma} \rightarrow \mathcal{C}\Tilde{\mathcal{S}_\Gamma}$ is the projection map, then there exists a constant $\dd$ such that for any $t,$ we have $d(p(\go),p(b(t)) \geq \frac{t}{\dd\kappa(t)}-2.$
\end{proposition}

\begin{proof}

The idea of the proof is that among all hyperplanes crossed by a geodesic ray $b$, only strongly separated hyperplanes will contribute to distance in the contact graph. For any given $t>0$, the number of such hyperplanes is asymptotically $\frac{t}{\kappa(t)}$.

Since $b$ is $\kappa$-contracting with constant, using Corollary \ref{hyperplane crossings in RAAG}, there exists $\cc$, and an infinite sequence of  hyperplanes

$$h_1,h_2,\dots, h_i, \dots$$ at points $b(t_i)$ such that $h_i,h_{i+1}$ are strongly separated and \[d(b(t_i), b(t_{i+1})) \leq \cc \kappa(t_{i+1}).\] Since each $h_i, h_{i+1}$ are strongly separated, every $h_i$ is strongly separated from $h_j$ by Theorem 2.5 in \cite{Behrstock2011}. For each $t \in \mathbb{R}$, there exists $i$ such that $b(t)$ is between $b(t_{i-1})$ and $b(t_{i})$. Thus

\begin{align*}
 d(\go, b(t)) &\leq d(\go, b(t_1))+d(b(t_1), b(t_2))+...+d(b(t_{i-1}), b(t_{i}))\\
                                                               & \leq \cc\kappa(t_1)+\cc \kappa(t_2)+...+\cc \kappa(t_{i})\\
                                                               & \leq \cc \kappa(t_{i})+\cc \kappa(t_{i})...+\cc \kappa(t_{i}) \qquad \text{($\kappa$ is monotone nondecreasing)}\\
                                                               &=i\cc \kappa(t_{i}).
\end{align*}  
This implies that $i \geq \frac{d(\go, b(t))}{\cc \kappa(t_{i})}$, where $t_{i-1} \leq t \leq t_i.$ Applying Lemma~\ref{lem:projection length} to the geodesic $b$ and the hyperplanes $h_1,h_2,..,h_i$ we get that $d(p(\go), p(b(t))) \geq i-2$. Using Proposition \ref{lemma: constants} there exists a constant $\dd'$ such that  $\kappa(t_i) \leq \dd' \kappa(t)$.  Using the above, we get that
 \[d(p(\go), p(b(t))) \geq i-2 \geq \frac{d(\go,b(t))}{\cc\kappa(t_i)}-2 = \frac{t}{\cc\kappa(t_i)}-2 \geq \frac{t}{\cc\dd'\kappa(t)}-2.\]

\end{proof}

\subsection{Bounding the well-separating constant in a factored CAT(0) cube complex} \label{subsec: Gen's graph}

Lastly, we prove that if a CAT(0) cube complex is equipped with a \emph{factor system}, then there is a uniform bound to the 
well-separation constant. Factor systems were introduced by \cite{Behrstock2017}. The Salvetti complexes are the main examples of CAT(0) cube complexes with a factor system, which we define now:

\begin{definition}[Parallel subcomplexes]\label{parallel} Let $X$ be a finite dimensional CAT(0) cube complex. Two convex subcomplexes $H_1, H_2$ are said to be \emph{parallel} if for any hyperplane $h,$ we have $h \cap H_1 \neq \emptyset$ if and only if $h \cap H_2 \neq \emptyset$.
\end{definition}

\begin{definition}[Factor systems \cite{Behrstock2017}] \label{def:factorsystem}Let $X$ be a finite dimensional CAT(0) cube complex. A \emph{factor system}, denoted $\mathfrak{F}$, is a collection of subcomplexes 
of $X$ such that:

\begin{enumerate}
    \item $X \in \mathfrak{F}$.
    \item Each $F \in \mathfrak{F}$ is a nonempty convex subcomplex of $X$
    \item There exists $\cc_{1} \geq 1$ such that for all $x \in X^{(0)}$ at most $\cc_{1}$ elements of $\mathfrak{F}$ contains $x.$
    \item Every nontrivial convex subcomplex paralell to a combinatorial hyperplane of $X$ is in $\mathfrak{F}$.
    \item There exists $\cc_{2}$ such that for all $F, F' \in \mathfrak{F}$, either $\sg_{F}(F') \in \mathfrak{F}$ or $$\text{diam}(\sg_{F}(F')) \leq \cc_{2}$$.
\end{enumerate}

\end{definition}

We also use the following generalization of contact graph introduced in \cite{HypInCube}: 

\begin{definition}[The well-separation space] \label{def:Morse-detecting space}
Let $X$ be a finite dimensional CAT(0) cube complex. For each integer $k$, the \emph{$k$-well-separation space}, denoted by $Y_k$ is defined to be the set whose elements are the vertices of $X$, with the following distance function. For $x,y \in X^{(0)}$, the $k$-distance between $x,y$ denoted by $d_k(x,y)$ is defined to be the cardinality of the maximal collection of $k$-well-separated hyperplanes separating $x,y$.

\end{definition}

It was shown in  \cite{HypInCube} that for any CAT(0) cube complex $X$, the $k$-well-separation space $(Y_k,d_k)$ is hyperbolic.

 \begin{proposition}[{\cite[Proposition 6.53]{HypInCube}}]\label{prop: Gen's metric is hyperbolic} Let $X$ be a finite dimensional CAT(0) cube complex and let $Y_k$ be the $k$-well-separation space as in Definition \ref{def:Morse-detecting space}. For any non-negative integer $k$, the metric space $(Y_k,d_k)$ is $9(k+2)$ hyperbolic.
\end{proposition}

\begin{proposition}[{\cite[Proposition 2.2]{HypInCube}}]\label{prop: Projections between subcomplexes}

Let $X$ be a CAT(0) cube complex and let $A,B \subseteq X$ be two convex subcomplexes. The hyperplanes intersecting $\sg_B(A)$ are precisely those which intersect both $A$ and $B.$
\end{proposition}

%
%
%

The proof of the following lemma is exactly the same as that of Lemma 6.53 in the first arXiv version of \cite{Gen-arXiv}. However as the Genevois noted, it is only suitable under the assumption that $X$ is a cocompact CAT(0) cube complex with a factor system. For completeness, we present it here.

\begin{lemma} \label{lem: well-separated implies L-well-separated}
Let X be a cocompact CAT(0) cube complex with a factor system. There exists a constant $L \geq 0$ such that any two hyperplanes of $X$
 either are $L$-well-separated or are not well-separated.
\end{lemma}

\begin{proof}
Let $G$ be the group acting cocompactly on $X$. Suppose for the sake of contradiction that there exists an increasing sequence of integers ($r_{i}$) and a sequence of pairs of hyperplanes $(h_{i}, u_{i})$ such that, for every $i \geq 1$, the hyperplanes
$h_{i}$ and $u_{i}$ are $r_{i}$-well-separated but not $(r_{i}-1)$-well-separated. Since $X$ is cocompact, up to passing to a subsequence of $h_i,$ there exists a hyperplane $h$ such that $h_{i}=g_ih$,   for $g_i \in Stab(h).$ Applying $g^{-1}_i$ to the sequence of pairs $(h_{i},u_{i})$, we obtain that the sequence of pairs $(h, g^{-1}_iu_{i})$ must also be $r_{i}$-well-separated but not $(r_{i}-1)$-well-separated. We will denote the hyperplanes $g^{-1}_iu_{i}$ by $v_i$. Now, let $x_i \in \sg_h(v_i)$. Since the stabiliser $Stab(h) \subset Isom(X)$ acts cocompactly on the carrier $N(h)$, there exists a combinatorial ball of a finite radius in $h$, call it $B$, and infinitely many  $g'_i \in G$ such that $g'_ix_i \in B$. As $B$ contains only finitely many vertices, up to passing to a subsequence, we have $g'_ix_i=x$ for some $x \in B.$ Thus, the hyperplanes $(h,g'_iv_i)$ are also $r_{i}$-well-separated but not $(r_{i}-1)$-well-separated. We will denote the hyperplanes $g'_iv_i$ by $w_i$. Thus, $\sg_{h}(w_{i})$ contains $x$ for every $i \geq 1$. Since the hyperplanes $(h,w_{i})$ are $r_{i}$-well-separated but not $(r_{i} - 1)$-well-separated,  we deduce from Proposition \ref{prop: Projections between subcomplexes} that the maximal cardinality of hyperplanes intersecting $\sg_{h}(w_{i})$ but not containing facing triple is precisely $r_i$.  Since $r_{i}$ is strictly increasing, we have $\sg_{h}(w_{i}) \neq \sg_{h}(w_{j})$ for all $i \neq j$. On the other hand, since $X$ is a CAT(0) cube complex with a factor system $FS$, using condition 3 of Definition \ref{def:factorsystem} there exists a constant $\cc$ such that at most $\cc$ elements of $\sg_h(w_i)$ are contained in $FS$. Therefore, condition 5 of Definition \ref{def:factorsystem} implies that there exists a constant $\cc'$ such that  for infinitely many $i,$ we have diam$(\sg_h(w_i)) \leq \cc'$, which is a contradiction.

\end{proof}

Therefore, in light of Lemma \ref{lem: well-separated implies L-well-separated}, we have the following. 

\begin{corollary} \label{cor: hyperplane charachterization for cocompact cube complexes}

Let $X$ be a finite dimensional cocompact CAT(0) cube complex with a factor system. There exists a constant $k$ such that a geodesic ray $b \in X$ is $\kappa$-contracting if and only if there exists a constant $\cc$ such that $b$ crosses an infinite sequence of hyperplanes $h_1, h_2,...$ at points $b(t_i)$ satisfying: 

\begin{enumerate}
    \item $d(t_i,t_{i+1}) \leq  \cc\kappa(t_{i+1}).$
    \item $h_i,h_{i+1}$ are $k$-well-separated.
\end{enumerate}
\end{corollary}

\begin{proof}
This is immediate by applying Lemma \ref{lem: well-separated implies L-well-separated} to Theorem \ref{thm: main theorem section 4}.
\end{proof}

As an application of Corollary \ref{cor: hyperplane charachterization for cocompact cube complexes}, we show that for any CAT(0) cube complex with a factor system, $\kappa$-contracting geodesic rays project to subsets of infinite diameter in the well-separation space. In order to make this precise, we define a projection map $p:X \rightarrow Y_k$ as in the following remark.

\begin{remark} \label{remark: projection to vertices}

For each $x \in X,$ let $n$ be the largest integer so that $x \in [0,1]^n$. We define $q(x)$ to be a $0$-cell (a vertex) of the cube $[0,1]^n$. Also, remark that since $X$ is a finite dimensional cube complex, there exists a constant $\dd$ depending only on the dimension of $X$ such that for any $x,y \in X$, if $l_k(x,y)$ is the cardinality of the maximal collection of $k$-well-separated hyperplanes separating $x,y$, then we have:

\begin{center}
$ d_k(q(x),q(y))-\dd\leq l_k(x,y) \leq d_k(q(x),q(y))+\dd. $
\end{center}
\end{remark}

Now, we prove Corollary \ref{cor:final statement of the intro} from the introduction.

\begin{corollary} \label{cor: progress in the well-separation space } Let $X$ be a finite dimensional cocompact CAT(0) cube complex with a factor system. There exists constants $\dd_1,\dd_2$ and an integer $k$ such that for any $\kappa$-contracting geodesic ray $b$ with $q(b(0))=\go,$ the projection map $q:X \rightarrow Y_k$ satisfies $$d_k(q(\go)), q(b(t))) \geq \frac{t}{\dd_1\kappa(t)}-\dd_2.$$

\end{corollary}

\begin{proof} Let $b$ be some $\kappa$-contracting geodesic ray. By Corollary \ref{cor: hyperplane charachterization for cocompact cube complexes}, there exists an integer $\cc$ such that $b$ crosses an infinite sequence of hyperplanes $h_i$ at points $t_i$ satisfying $d(t_i,t_{i+1}) \leq  \cc\kappa(t_{i+1})$ such that $h_i,h_{i+1}$ are $\cc$-well-separated. Notice that for each $i$, since $h_i \cap h_{i+1}= h_{i+1} \cap h_{i+2} =\emptyset$, the hyperplane $h_i$ separates $h_{i-1}$ from $h_{i+1}$. Thus, every hyperplane intersecting both $h_i,h_{i+2}$ must also intersect $h_i, h_{i+1}$. Therefore, $h_i, h_{i+2}$ are also $\cc$-well-separated. Indeed, by induction, this implies that for any $j<k,$ the hyperplanes $h_j,h_k$ are also $\cc$-well-separated. Now, let $t>0$ and let $i$ be so that $t_{i-1} \leq t \leq t_{i}.$

\begin{align*} 
 d(\go, b(t)) &\leq d(\go, b(t_1))+d(b(t_1), b(t_2))+...+d(b(t_{i-1}), b(t_{i}))\\
                                                               & \leq \cc\kappa(t_1)+\cc \kappa(t_2)+...+\cc \kappa(t_{i})\\
                                                               & \leq \cc \kappa(t_{i})+\cc \kappa(t_{i})...+\cc \kappa(t_{i}) \qquad \text{($\kappa$ is monotone nondecreasing)}\\
                                                               &=i\cc \kappa(t_{i}).
\end{align*}

This gives that $i \geq \frac{t}{\cc\kappa(t)}$ where $i-1$ is a lower bound on the number of $\cc$-well-separated hyperplanes separating $b(0)$ and $b(t)$. Let $k=\cc$, using Remark \ref{remark: projection to vertices}, there exists a constant $\dd$ so that $d_k(q(b(0)),q(b(t))) \geq i-\dd \geq \frac{t}{\cc\kappa(t_i)}-\dd \geq \frac{t}{k\kappa(t_i)}-\dd.$ Using Proposition \ref{lemma: constants}, there exists a constant $\dd'$ such that $\kappa(t_i) \leq \dd' \kappa(t),$ therefore, we have $d_k(q(\go),q(b(t))) \geq i-\dd \geq \frac{t}{\cc\kappa(t_i)}-\dd \geq \frac{t}{\cc\kappa(t_i)}-\dd \geq \frac{t}{\cc\dd'\kappa(t)}-\dd$.

\end{proof}

\bibliography{main}{}
\bibliographystyle{alpha}

\end{document}